\documentclass[11pt]{amsart}
\usepackage{amssymb}
\usepackage{amsmath}
\usepackage{amsfonts}
\usepackage{graphicx}

\usepackage[total={17cm,22cm},top=2.5cm, left=2.3cm]{geometry}
\usepackage{hyperref}
    \usepackage{aeguill}
    \usepackage{type1cm}

\usepackage[shortlabels]{enumitem}

\newtheorem{thm}{Theorem}[section]

\newtheorem{lem}[thm]{Lemma}
\newtheorem{rem}[thm]{Remark}

\newtheorem{prop}[thm]{Proposition}

\newtheorem{cla}[thm]{Claim}

\newtheorem{prob}[thm]{Problem}
\numberwithin{equation}{section}

\date{\today}

\usepackage{amsxtra}
\usepackage{eucal}
\usepackage{mathrsfs}
\usepackage{longtable}
\usepackage{caption}

\usepackage{hyperref}
\usepackage{aeguill}
\usepackage{type1cm}

\makeatletter

\newcommand{\Rmnum}[1]{\expandafter\@slowromancap\romannumeral #1@}
\makeatother

\def\D{\mathbb{D}}

\def\M{{\mathcal M}}
\def\O{{\mathcal O}}

\usepackage[usenames,dvipsnames]{color}

\usepackage[dvipsnames]{xcolor}

\begin{document}

\title{Essentially commuting dual truncated Toeplitz operators}

\author[Chongchao Wang]{Chongchao Wang\textsuperscript{1}}
\address{\textsuperscript{1} College of Mathematics and Statistics, Chongqing University, Chongqing, 401331, P. R. China}
\email{chongchaowang@cqu.edu.cn}
\author[Xianfeng Zhao]{Xianfeng Zhao\textsuperscript{2}}
\address{\textsuperscript{2} College of Mathematics and Statistics, Chongqing University, Chongqing, 401331, P. R. China}
\email{xianfengzhao@cqu.edu.cn}
\author[Dechao Zheng]{Dechao Zheng\textsuperscript{3}}
\address{\textsuperscript{3} Center of Mathematics, Chongqing University, Chongqing, 401331, P. R. China and  Department of Mathematics, Vanderbilt University, Nashville, TN 37240, United States}
\email{dechao.zheng@vanderbilt.edu}

\keywords{Hardy space, dual truncated Toeplitz operator, essentially commuting}

\subjclass[2010]{47B32, 47B35}


\begin{abstract}
In this paper, we completely characterize when two dual truncated  Toeplitz operators are essentially commuting and when the semicommutator  of two dual truncated Toeplitz operators is compact. Our main idea  is to study dual truncated Toeplitz operators via  Hankel operators,  Toeplitz operators  and function algebras.
\end{abstract}

\maketitle

\section{introduction}
Let $\D$ be the open unit disk and $\partial \mathbb D$ be its boundary. Let $L^2$ denote the Lebesgue space of square integrable functions on the unit circle $\partial \mathbb D$. The Hardy space $H^2$ is the closed subspace of $L^2$, which is spanned by the space of analytic polynomials. Thus there is an orthogonal projection $P$ from $L^2$ onto $H^2$. For $\varphi$ in $L^\infty$, the space of essentially bounded  measurable functions on  $\partial \mathbb D$, the Toeplitz operator $T_\varphi$ and the Hankel operator $H_\varphi$ with symbol $\varphi$ on $H^2$ are defined by
$$T_\varphi f=P(\varphi f)$$
and
$$H_\varphi f=(I-P)(\varphi f)$$
for $f\in H^2$, respectively. Moreover, the dual Toeplitz operator $S_\varphi$ on $(H^2)^{\bot}$  is defined by
$$S_\varphi h= (I-P)(\varphi h),  \ \ \ \  h\in (H^2)^{\bot}.$$
For more information on the topics of Toeplitz and Hankel operators we refer to \cite{Dou, Zhu}.

Let $T_z^*$ be the adjoint of the forward shift operator $T_z$. Suppose that $u$ is a nonconstant inner function. The invariant subspace for $T_z^*$ $$K_{u}^2=H^2\ominus uH^2$$ is called  the model space \cite{G}. Let $P_{u}$ be the orthogonal projection from $L^2$ onto $K_{u}^2$. For $\varphi\in L^2$,  the dual truncated Toeplitz operator $D_\varphi$ with symbol $\varphi$ on the orthogonal complement of $K_u^2$ is densely defined by
$$D_{\varphi}f=(I-P_u)(\varphi f)$$ on the subspace $(K_{u}^2)^\perp \cap L^\infty$ of  $(K_{u}^2)^\perp=L^2\ominus K_{u}^2$. Noting that $L^2=H^2\oplus \overline{zH^2}$ and $K_{u}^2=H^2\ominus uH^2$, we obtain $$(K_{u}^2)^\perp=uH^2\oplus \overline{zH^2},$$ and moreover,
$$P_{u}=P-M_{u}PM_{\overline{u}}$$
and $$ I-P_{u}=M_{u}PM_{\overline{u}}+(I-P),$$
where $M_u$ is the multiplication operator on $H^2$ with symbol $u$.

Toeplitz operators and Hankel operators have played an especially  important role in function theory and operator theory. There are
many fascinating problems about those two classes of operators.  The essentially commuting problem of two bounded linear operators  arises from studying Fredholm theory of operators on a Hilbert space. The answer to the commuting problem for two Topelitz operators on the Hardy space was obtained by Brown and Halmos \cite{BrHa} in 1964, which states that two Toeplitz operators are commuting if and only if either both symbols of these operators are analytic, or both symbols of these operators are co-analytic, or a nontrivial linear combination of their symbols
 is  constant. Axler and \v{C}u\v{c}kvoi\'{c} obtained the analogous result for Toeplitz
operators with bounded harmonic symbols on the Bergman space of the unit disk \cite{AC}. Using some techniques in multiple complex-variable functions, Ding, Sun and Zheng \cite{DZ} established a necessary and sufficient condition for two Toeplitz operators  to be commuting on the Hardy space over the bidisk.

 The problem  of when the commutator or semicommutator of two operators  is compact on function spaces  has been investigated by many people.   The beautiful Axler-Chang-Sarason-Volberg theorem (\cite{ACS}, \cite{V}) states that   the semicommutator $ T_{f}T_{g}-T_{fg}$  of two
Hardy Toeplitz operators $T_f$ and $T_g$ is compact if and only if  either $\overline{f}$ or $g$ is in $H^\infty$ on each \emph{support set} (which will be introduced in the next section). An elementary characterization for the compactness of the semicommutator of two Hardy Toeplitz operators in terms of Hankel operators
was obtained by Zheng \cite{Zheng}. The compactness for the semicommutator of two Toeplitz operators on other analytic function spaces has been studied in \cite{Gu}, \cite{GSZ} and \cite{Zheng1}.

 In 1999, Gorkin and Zheng \cite{GpZ}   completely characterized the compact commutator $T_fT_g-T_gT_f$ of two Toeplitz operators on the Hardy space in terms of Douglas algebras or support sets.   More precisely, the characterization in \cite{GpZ} can be stated as follows: two Toeplitz operators are essentially commuting if and only if either the restrictions of their symbols on each support set $S$ are in $H^\infty|_{S}$, or the restrictions of the conjugations of their symbols on each $S$ belong to $H^\infty|_{S}$, or a nontrivial linear combination of the restrictions of their symbols on  each support set $S$ is constant. The essentially commuting problem for Toeplitz operators with bounded harmonic symbols on the Bergman space was solved by Stroethoff \cite{Str} in 1993.

Dual Toeplitz operators on the orthogonal complement of the Bergman space were studied in \cite{StZ}.
Dual truncated Toeplitz operator is a new class of operators on the  orthogonal complement of the  model space, which was first introduced in \cite{DS}. In \cite{Ca}, asymmetric dual truncated Toeplitz operators acting between the
orthogonal complements of two (eventually different) model spaces were
introduced.  Although these operators differ in many ways from Toeplitz operators on the Hardy space, they do have some of the same interesting properties, see \cite{DS} and \cite{DSQ} for more information. In the present paper, we focus on the following  problems:

\begin{prob}
When is the commutator $[D_f, D_g]=D_fD_g-D_gD_f$ of two dual truncated Toeplitz operators $D_f$ and $D_g$ with $f$ and $g$ in $L^\infty$
compact?
\end{prob}

\begin{prob}
When is   the semicommutator $[D_f,D_g)=D_fD_g-D_{fg}$ of two dual truncated Toeplitz operators $D_f$ and $D_g$ with $f$ and $g$ in $L^\infty$
compact?
\end{prob}

In order to study  the dual truncated Toeplitz operators, we use the useful matrix representation for the
 dual truncated Toeplitz operator to establish a connection between the Toeplitz operator,  Hankel operator and dual  truncated   Toeplitz operator. Then the above essentially commuting (semicommuting) problem can be reduced to the study of the compactness of products of Toeplitz, Hankel and dual Toeplitz operators.
 The  difficult part in this paper is characterizing the compactness of the sum of the four products of Toeplitz, Hankel and dual Toeplitz operators. Our main idea here is to study dual truncated Toeplitz operators via the characterization for the essentially commuting Hankel and Toeplitz operators \cite{GkZ1} and function algebras. The first main result in this paper is the following theorem.

\begin{thm}\label{MR} Let $u$ be a nonconstant inner function and $f,g\in L^\infty$. The commutator $[D_f,D_g]$ is compact if and only if for each support set $S$, one of the following holds: \\
  $(1)$ $f|_S$, $g|_S$, $\left((u-\lambda)\overline{f}\right)|_{S}$ and $\left((u-\lambda)\overline{g}\right)|_{S}$ are in $H^\infty|_{S}$ for some constant $\lambda$; \\
  $(2)$ $\overline{f}|_S$, $\overline{g}|_S$, $\left((u-\lambda)f\right)|_{S}$ and $\left((u-\lambda)g\right)|_{S}$ are in $H^\infty|_{S}$ for some constant  $\lambda$; \\
  $(3)$ there exist constants $a$, $b$, not both zero, such that $(af+bg)|_{S}$ is a constant.
\end{thm}
The above theorem  is  analogous to the characterization when
two Toeplitz operators are essentially commuting on the Hardy space \cite[Theorem 0.8]{GpZ}.

The second main result of our paper is the following characterization on the compactness of the semicommutator of two dual truncated Toeplitz operators.

\begin{thm}\label{MR2} Let $u$ be a nonconstant inner function and $f,g\in L^\infty$. The semicommutator $[D_f,D_g)$ is compact if and only if for each support set $S$, one of the following holds: \\
  $(1)$ $f|_S$, $g|_S$, $\left((u-\lambda)\overline{f}\right)|_{S}$, $\left((u-\lambda)\overline{g}\right)|_{S}$ and
  $\left((u-\lambda)\overline{fg}\right)|_{S}$ are in $H^\infty|_{S}$ for some constant $\lambda$; \\
  $(2)$ $\overline{f}|_S$, $\overline{g}|_S$, $\left((u-\lambda)f\right)|_{S}$, $\left((u-\lambda)g\right)|_{S}$ and
  $\left((u-\lambda)fg\right)|_{S}$ are in $H^\infty|_{S}$ for some constant  $\lambda$; \\
  $(3)$ either $f|_S$ or $g|_S$ is a constant.
\end{thm}

Theorem \ref{MR2} is analogous to the characterization for the compactness of the semicommutator of two Hardy Toeplitz operators (see \cite{ACS, V}).

As the proof of Theorem \ref{MR} is long, it is divided into the necessary part in Section 3 and the sufficient part in Section 4. We will present  the details for the proof of the necessary part and the sufficient part of Theorem \ref{MR2}  in Sections 5 and 6, respectively.

\section{ Notations and preliminaries}
In this section, we introduce some notations and include some important lemmas.
Let us  begin with the following matrix representation for the dual truncated Toeplitz operator on the space $(K_{u}^2)^{\perp}$, see \cite[Lemma 2]{SQD} for the details.
\begin{lem}\label{1a}
Suppose that $\varphi \in L^\infty$. The dual truncated Toeplitz operator $D_\varphi$ on $(K_{u}^2)^{\perp}$ is unitarily equivalent to the following $(2\times 2)$ operator matrix
\[
\left(\begin{array}{cccc}
    T_{\varphi} & H_{u\overline{\varphi}}^* \\
    H_{u\varphi} & S_\varphi
\end{array}\right)
\]
on the space $L^2=H^2\oplus \overline{zH^2}$. Moreover, the unitary operator here is given by
\[
U=\left(\begin{array}{cccc}
    M_u & 0\\
    0 & I
\end{array}\right).
\]
\end{lem}
In view of the matrix representation in the above lemma, the essentially commuting  problem for two dual truncated Toeplitz operators can be easily transformed into the compactness of the following four classical operators.
\begin{lem}\label{1b}
Suppose that $u$ is a nonconstant inner function and $f,g\in L^\infty$. Then the commutator $D_fD_g-D_gD_f$ is compact if and only if
$$T_fT_g+H_{u\overline{f}}^*H_{ug}-T_gT_f-H_{u\overline{g}}^*H_{uf},$$
$$T_fH_{u\overline{g}}^*+H_{u\overline{f}}^*S_g-T_gH_{u\overline{f}}^*-H_{u\overline{g}}^*S_f,$$
$$H_{uf}T_g+S_fH_{ug}-H_{ug}T_f-S_gH_{uf}$$ and
$$H_{uf}H_{u\overline{g}}^*+S_fS_g-H_{ug}H_{u\overline{f}}^*-S_gS_f$$
are  compact.
\end{lem}
\begin{proof}
Let
$$T_1=T_fT_g+H_{u\overline{f}}^*H_{ug}-T_gT_f-H_{u\overline{g}}^*H_{uf},$$
$$T_2=T_fH_{u\overline{g}}^*+H_{u\overline{f}}^*S_g-T_gH_{u\overline{f}}^*-H_{u\overline{g}}^*S_f,$$
$$T_3=H_{uf}T_g+S_fH_{ug}-H_{ug}T_f-S_gH_{uf}$$
and
$$T_4=H_{uf}H_{u\overline{g}}^*+S_fS_g-H_{ug}H_{u\overline{f}}^*-S_gS_f.$$
Then we have by Lemma \ref{1a} that
\begin{align*}
&U^*(D_fD_g-D_gD_f)U\\
&=\left (\begin{matrix}
  T_{f} & H_{u\overline{f}}^* \\
    H_{uf} & S_f
  \end{matrix}\right)\left (\begin{matrix}
    T_{g} & H_{u\overline{g}}^* \\
    H_{ug} & S_g
  \end{matrix}\right )-\left(\begin{matrix}
   T_{g} & H_{u\overline{g}}^* \\
    H_{ug} & S_g
  \end{matrix}\right )\left (\begin{matrix}
  T_{f} & H_{u\overline{f}}^* \\
    H_{uf} & S_f
  \end{matrix}\right)\\
  &=\left (\begin{matrix}
  T_1  & T_2 \\
   T_3 & T_4
  \end{matrix}\right).
 \end{align*}
Denote the above operator matrix by $T$. So we need only to show that $T$ is compact if and only if $T_1$, $T_2$, $T_3$ and $T_4$ are  compact. For completeness, we include a proof here for this well-known result about operator matrix.

Suppose that $T$ is compact. Let $\{f_n\}_{n=1}^\infty$ be a sequence in $H^2$   converging weakly to 0. Then $\binom{f_n}{0}$ converges weakly to 0. Since $T$ is compact, we have
\begin{align*}
\begin{pmatrix} T_1 & T_2\\ T_3 & T_4\end{pmatrix}\binom{f_n}{0}
=\binom{T_1f_n}{T_3f_n}\rightarrow 0
\end{align*}
in $L^2$-norm, which gives that $\|T_1f_n\|_2\rightarrow 0$ and $\|T_3f_n\|_2\rightarrow 0$ as $n\rightarrow \infty$.

For any sequence $\{g_n\}_{n=1}^\infty$ in $\overline{zH^2}$ which converges weakly to 0, we similarly obtain that $\|T_2g_n\|_2\rightarrow 0$ and $\|T_4g_n\|_2\rightarrow 0$ as $n\rightarrow \infty$. Therefore, we have $T_1$, $T_2$, $T_3$ and $T_4$ are   compact.

Conversely we assume that $T_1$, $T_2$, $T_3$ and $T_4$ are compact. Let $\{h_n\}_{n=1}^\infty$ be a sequence in $L^2=H^2\oplus \overline{zH^2}$ such that $h_n$  converges weakly  to 0. Let $$h_n=\binom{f_n}{g_n},$$ where $f_n\in H^2$, and $g_n\in\overline{zH^2}$. Then both $f_n$ and $g_n$ converge weakly to $0$ as $n\rightarrow \infty$.

Noting that
\begin{align*}
\bigg\|\begin{pmatrix} T_1 & T_2\\ T_3 & T_4\end{pmatrix}\binom{f_n}{g_n} \bigg\|_2
&=\left\|\binom{T_1f_n+T_2g_n}{T_3f_n+T_4g_n}\right\|_2 \\
&\leqslant \|T_1f_n\|_2+\|T_2g_n\|_2+\|T_3f_n\|_2+\|T_4g_n\|_2,
\end{align*}
we conclude by the compactness of $T_1, T_2, T_3$ and $T_4$ that
\begin{align*}
\begin{pmatrix} T_1 & T_2\\ T_3 & T_4\end{pmatrix} \binom{f_n}{g_n}\rightarrow 0
\end{align*}
in $L^2$-norm, which implies that $T$ is compact. This completes the proof.
\end{proof}
Using the same method as in the proof of Lemma \ref{1b}, we obtain a similar conclusion for the compactness of the semicommutator  $[D_f,D_g)$.
\begin{lem}\label{semi-c}
Suppose that $u$ is a nonconstant inner function and $f,\ g\in L^{\infty}$. Then
the semicommutator $D_fD_g-D_{fg}$  is compact if and only if
$$T_fT_g+H_{u\overline{f}}^{*}H_{ug}-T_{fg},$$
$$T_fH_{u\overline{g}}^{*}+H_{u\overline{f}}^{*}S_g-H_{u\overline{fg}}^{*},$$
$$H_{uf}T_g+S_{f}H_{ug}-H_{ufg}$$
and
$$H_{uf}H_{u\overline{g}}^{*}+S_{f}S_{g}-S_{fg}$$
 are compact.
\end{lem}

To study the compactness of  products of Hankel and Toeplitz operators on the Hardy space, the following operator $V$ is very useful.
Define the operator  $V:\ L^2\rightarrow L^2$ by $$Vf(z)=\overline{zf(z)}, \ \ \  f\in L^2, \ z\in \partial \mathbb D.$$
It is easy to check that  $V$ is anti-unitary and moreover,
 $$V=V^{-1}=V^*$$
 on  $L^2$. For a general anti-linear operator $V$, $V^*$ is the anti-linear operator  defined via the property
 $$\overline{\langle Vf, g\rangle}=\langle f, V^*g\rangle$$
for $f$ and $g$ in $L^2$.

 We will show  in the next lemma that the operator $V$ and the Hardy  projection $P$ satisfy the following equation.
\begin{lem}\label{2I}
For $f\in L^2$, then $$VP(f)=(I-P)V(f).$$
\end{lem}
\begin{proof}
For any  $f$ in $L^2$, we write $f=f_{+}+f_{-},$
where $f_{+}=Pf$ and $f_{-}=(I-P)f$. Then  we have
\begin{align*}
VP(f)(w)&=Vf_{+}(w)\\
     &=\overline{w}\overline{f_{+}(w)}\\
     &=\overline{w}\overline{f_{+}(w)}+(I-P)(\overline{w}\overline{f_{-}(w)})\\
     &=(I-P)(\overline{w}\overline{f_{+}(w)}+\overline{w}\overline{f_{-}(w)})\\
     &=(I-P)(\overline{w}\overline{f(w)})\\
     &=(I-P)V(f)(w)
\end{align*}
for each $w\in \partial\mathbb D$, to complete the proof.
\end{proof}
\begin{rem}\label{rem1}
Observe that Lemma \ref{2I} easily leads to the following two relations:
$$VH_{\varphi}=H_{\varphi}^*V\  \ \ \  \mathrm{and} \ \ \ \ S_{\varphi}V=VT_{\overline{\varphi}},$$
 which will be used repeatedly later on.
\end{rem}

For $x$ and $y$ in $L^2$, we use $x\otimes y$ to denote the following rank-one operator: for $f\in L^2$,
$$(x\otimes y)(f)=\langle f, y\rangle x.$$

It is well-known that the operator norm of the above rank-one operator is given by  $\|x\otimes y\|=\|x\|_{2}\cdot \|y\|_2$.
The following two lemmas about the Toeplitz and Hankel operators on $H^2$ established in \cite[Lemmas 1 and 2]{Zheng} are useful tools to study the compactness of the product of Hankel operators and compact operators in the Toeplitz algebra.
\begin{lem}\label{2a}
Let $f$ and $g$ be in $L^2$, and $z\in \D$. Then
$$H_{f}^*H_g-T_{\phi_z}^*H_{f}^*H_gT_{\phi_z}=V\big[(H_fk_z)\otimes(H_gk_z)\big]V^*.$$
\end{lem}

\noindent Here $$k_z(e^{i\theta})=\frac{\sqrt{1-|z|^2}}{1-\overline{z}e^{i\theta}}$$ is the normalized reproducing kernel for the Hardy space, and $\phi_z$ denotes  the M\"{o}bius map:
$$\phi_z(w)=\frac{z-w}{1-\overline{z}w} \ \ \ \ \  \big(z, w\in \mathbb D\big).$$

\begin{lem}\label{2b}
Let $K$ be a compact operator on $H^2$. Then we have
$$\lim_{|z|\rightarrow 1^-} \|K-T_{\phi_z}^*KT_{\phi_z}\|=0.$$
\end{lem}

 As in \cite{Gar}, a Douglas algebra is, by definition, a closed subalgebra of $L^\infty$ which contains $H^\infty$. As Douglas algebras play a prominent role in various problems on Toeplitz and Hankel operators, we need to review some important properties of them. Observe that $H^\infty$ is a commutative Banach algebra, we can identify the maximal ideal space $\mathcal{M}(H^\infty)$ as the set of multiplicative linear functionals on $H^\infty$. Endowed with the weak star topology it inherits as a subset of the dual space of $H^\infty$, $\mathcal{M}(H^\infty)$ is a compact Hausdorff space. Identifying a point in the open unit disk $\D$ with the functional of evaluation at this point, we may regard the disk $\D$ as a subset of $\mathcal{M}(H^\infty)$. Using the Gelfand transform we regard every function in $H^\infty$ as a continuous function on $\mathcal{M}(H^\infty)$. The deepest result concerning $\mathcal{M}(H^\infty)$ is the famous corona theorem of Carleson, stating that $\D$ is dense in $\mathcal{M}(H^\infty)$ under the weak star topology (for details, see \cite{Dur} and \cite{Gar}).

It is a consequence of the Gleason-Whitney theorem that the maximal ideal space of a Douglas algebra $B$ is a naturally imbedded in $\mathcal{M}(H^\infty)$. Thus we may identify the maximal ideal space $\mathcal{M}(H^\infty+C)$ of the Sarason algebra $H^\infty+C$ with a subset of $\mathcal{M}(H^\infty)$, where $C$ is the algebra of continuous functions on $\partial\mathbb D$. A subset of $\mathcal{M}(L^\infty)$ will be a support set if it is the (closed) support of the representing measure for a functional in $\mathcal{M}(H^\infty+C)$, see \cite{Gar} and \cite{Hof} for more details. Let $m$ be in $\mathcal{M}(H^\infty+C)$ and let $d\mu_m$ denote the unique representing measure for $m$ with support $S_m$, i.e.,\\
(1)~ for all $f$ and $g$ in $H^\infty$, $$m(fg)=\int_{S_m}fg~d\mu_m=\bigg(\int_{S_m} f d\mu_m\bigg)\bigg(\int_{S_m} g d\mu_m\bigg); $$
(2)~ if $h\geqslant 0$ a.e. in $L^1(d\mu_m)$ such that $$\int_{S_m}fh~d\mu_m=\int_{S_m} f d\mu_m$$ for all $f\in H^\infty$, then we have $h=1$ a.e. $d\mu_m$.

Suppose that $m\in \mathcal{M}(H^\infty+C)$ and $z\mapsto \xi_z$ is a mapping from the unit disk $\mathbb D$ into some topological space $X$. Let $\eta$ be in $X$. We use the notation
$$\lim_{z\rightarrow m} \xi_z=\eta$$
to denote that for each open set $\mathcal{U}(\eta)\subset X$ containing $\eta$, there exists an open subset $\mathcal{O}(m)$ of $\mathcal{M}(H^\infty+C)$  containing $m$ such that $\xi_z\in \mathcal{U}$ for all $z\in \mathcal{O}(m)\cap \mathbb D$.

For a function $F$ on the disk $\D$ and $m$ in $\mathcal{M}(H^\infty+C)$, we say
$$\lim_{z\rightarrow m} F(z)=0$$
if for every net $\{z_\alpha\}\subset \D$ converging to $m$,
$$\lim_{z_\alpha\rightarrow m} F(z_\alpha)=0.$$
We shall emphasize here that we  deal with nets rather than sequences since the the topology of $\mathcal{M}(H^\infty+C)$ is not metrizable.

With the above notations and concepts about $H^2$ theory on a support set, we quote the following lemma  obtained in \cite[Lemmas 2.5 and 2.6]{GpZ}.
\begin{lem}\label{2e}
Let $f$ be in $L^\infty$ and $m\in \mathcal M(H^\infty+C)$. Denote the support set for $m$ by $S_m$. Then
the following three conditions are equivalent:\\
$(1)$ $f|_{S_m}\in H^\infty|_{S_m}$;\vspace{1.5mm}\\
$(2)$ $\lim\limits_{z\rightarrow m}\|H_{f}k_z\|_2=0$;\vspace{1.5mm}\\
$(3)$ $\varliminf\limits_{z\rightarrow m}\|H_{f}k_z\|_2=0$.
\end{lem}

\section{\label{S1} The necessary part of  Theorem \ref{MR}}
In this section, we assume that  $D_fD_g-D_gD_f$ is a compact operator. Recall that the four operators in  Lemma \ref{1b} are  compact. Now we are going to derive the necessary condition for the compactness of these four operators in terms of the boundary properties of the symbols $f$ and $g$.

In the following proposition, we establish a necessary condition for the compactness of  the first operator $T_fT_g+H_{u\overline{f}}^*H_{ug}-T_gT_f-H_{u\overline{g}}^*H_{uf}$  given in Lemma \ref{1b}.
\begin{prop}\label{2c}
Let $u$ be a nonconstant inner function, $f,g\in L^\infty$ and $m\in \mathcal M(H^\infty + C).$ Suppose that the operator  $$T_fT_g+H_{u\overline{f}}^*H_{ug}-T_gT_f-H_{u\overline{g}}^*H_{uf}$$ is compact. Then for the support set $S_m$ of $m$, one of following  conditions holds:\\
  $(1)$ both $f|_{S_m}$ and $g|_{S_m}$ are in $H^\infty|_{S_m}$;\\
  $(2)$ both $\overline{f}|_{S_m}$ and $\overline{g}|_{S_m}$ are in $H^\infty|_{S_m}$;\\
  $(3)$ there exist constants $a$, $b$, not both zero, such that $(af+bg)|_{S_m}$ is a constant.
\end{prop}

\begin{proof}
Suppose that $$T_fT_g+H_{u\overline{f}}^*H_{ug}=T_gT_f+H_{u\overline{g}}^*H_{uf}+K,$$  where $K$ is compact. Since $T_fT_g-T_gT_f=H_{\overline{g}}^*H_f-H_{\overline{f}}^*H_g$,
we have $$H_{\overline{g}}^*H_f-H_{\overline{f}}^*H_g=H_{u\overline{g}}^*H_{uf}-H_{u\overline{f}}^*H_{ug}+K.$$
By Lemmas \ref{2a} and \ref{2b}, we have

\begin{eqnarray*}
K-T_{\phi_z}^*KT_{\phi_z}\vspace{2mm}=V\left(H_{\overline{g}}k_z\otimes H_fk_z-H_{\overline{f}}k_z\otimes H_gk_z-H_{u\overline{g}}k_z\otimes H_{uf}k_z+H_{u\overline{f}}k_z\otimes H_{ug}k_z\right)V^*
\end{eqnarray*}
and
\begin{equation}\label{2.a}
\begin{array}{l}
H_{\overline{g}}k_z\otimes H_fk_z-H_{\overline{f}}k_z\otimes H_gk_z=H_{u\overline{g}}k_z\otimes H_{uf}k_z-H_{u\overline{f}}k_z\otimes H_{ug}k_z+\varepsilon(z),
\end{array}
\end{equation}
where the operator $\varepsilon(z)$ satisfies that $\lim\limits_{|z|\rightarrow 1^-}\|\varepsilon(z)\|=0.$

In the following, we still use the same notation $\varepsilon(z)$ to denote the various terms such that $$\|\varepsilon(z)\|\rightarrow 0\ \ \ \ \ (|z|\rightarrow 1^-)$$ for simplicity.

For $m\in \mathcal M(H^\infty+C)$, we write $$[f|_{S_m}],\ [g|_{S_m}]\in (L^\infty|_{S_m}) / (H^\infty|_{S_m}),$$
where $[f|_{S_m}]$ denotes the coset $\{f|_{S_m}+h|_{S_m}:h|_{S_m}\in H^{\infty}|_{S_m}\}.$
As $(L^\infty|_{S_m}) / (H^\infty|_{S_m})$ is a Banach space,  we  consider the following three cases:
\begin{enumerate}
      \item $\mathrm{dim}\left(\mathrm{span}\{[f|_{S_m}],\ [g|_{S_m}]\}\right)=0$;
      \item $\mathrm{dim}\left(\mathrm{span}\{[f|_{S_m}],\ [g|_{S_m}]\}\right)=1$;
      \item $\mathrm{dim}\left(\mathrm{span} \{[f|_{S_m}],\ [g|_{S_m}]\}\right)=2$.
\end{enumerate}

\textbf{Case 1.} If $\mathrm{dim}\left(\mathrm{span} \{[f|_{S_m}],\ [g|_{S_m}]\}\right)=0$, then $[f|_{S_m}]=\ [g|_{S_m}]=0$, which implies that $f|_{S_m},g|_{S_m}\in H^\infty|_{S_m}$.

\textbf{Case 2.} If $\mathrm{dim}\left(\mathrm{span} \{[f|_{S_m}],\ [g|_{S_m}]\}\right)=1$, we may assume that $[g|_{S_m}]\neq 0$. Then there is a constant $c$ such that $[f|_{S_m}]=c[g|_{S_m}]$. By Lemma \ref{2e}, now (\ref{2.a}) can be rewritten as follows:
\begin{equation*}
\begin{array}{l}
H_{\overline{g}}k_z\otimes H_{cg}k_z-H_{\overline{f}}k_z\otimes H_gk_z \vspace{2mm}\\
=H_{u\overline{g}}k_z\otimes H_{c ug}k_z-H_{u\overline{f}}k_z\otimes H_{ug}k_z+\varepsilon(z),
\end{array}
\end{equation*}
to obtain
\begin{eqnarray} \label{2.b}
H_{\overline{cg}-\overline{f}}k_z\otimes H_gk_z=H_{u(\overline{cg}-\overline{f})}k_z\otimes H_{ug}k_z+\varepsilon(z),
\end{eqnarray}
where $\varepsilon(z)$ satisfies that $\|\varepsilon(z)\| \rightarrow 0$ as $|z|\rightarrow 1^-$.

To derive the desired conclusions, we are going to discuss two cases. First, if $$\varliminf_{z\rightarrow m}\|H_{\overline{cg}-\overline{f}}k_z\|_2=0,$$ then $\left(\overline{cg}-\overline{f}\right)|_{S_m}\in H^\infty|_{S_m}$.
Since $\left[(f-cg)|_{S_m}\right]=0$, we see that $(f-cg)|_{S_m}$ must be a constant.

 Now we need to analyse the case of
$$\varliminf_{z\rightarrow m}\|H_{\overline{c}\overline{g}-\overline{f}}k_z\|_2> 0.$$
By (\ref{2.b}), we have
\begin{eqnarray*}
\langle H_{\overline{cg}-\overline{f}}k_z,H_{\overline{cg}-\overline{f}}k_z\rangle H_gk_z=\langle H_{u(\overline{cg}-\overline{f})}k_z,H_{\overline{cg}-\overline{f}}k_z\rangle H_{ug}k_z+\varepsilon(z).
\end{eqnarray*}
Thus there exists a constant $a(z)$ depending  on $z$ such that
$$H_{g}k_z=a(z)H_{ug}k_z+\varepsilon(z),$$
where $a(z)$ satisfies that
$$|a(z)|=\left|\frac{\langle H_{u(\overline{c}\overline{g}-\overline{f})}k_z,H_{\overline{c}\overline{g}-\overline{f}}k_z\rangle }{\|H_{\overline{c}\overline{g}-\overline{f}}k_z\|_2^2}\right|\leqslant |u|\leqslant 1$$
for all $z\in \mathcal{O}(m)\cap \mathbb D$, so $|a(z)|$ is uniformly bounded for $z\in \mathcal{O}(m)\cap \mathbb D$. Here and below, we use $\mathcal{O}(m)$ to denote some small neighbourhood of $m\in \mathcal{M}(H^\infty+C)$.

Applying the Bolzano-Weierstrass theorem, we can choose a constant $a$ with $|a|\leqslant 1$ such that $a$ is independent of $z$ and
$$H_{g}k_z=aH_{ug}k_z+\varepsilon(z).$$
Hence we have that
\begin{eqnarray*}
\lim\limits_{z\rightarrow m}\|H_{(1-au)g}k_z\|_2=0.
\end{eqnarray*}
Making a change of  variables  yields
$$\lim\limits_{z\rightarrow m}\big\|(I-P)\big[(1-au\circ\phi_z)(g\circ\phi_z)\big]\big\|_2=0.$$

Since $|a|\leqslant 1$ and $u$ is not a constant on $S_m$,  we have by \cite[Lemma 1]{GkZ2} that $(1-au)$ is an outer function on the support set $S_m$. For any
$\varepsilon > 0$, there exists a function $p\in H^\infty$ such that
$$\int_{S_m}\big|p(1-au)-1\big|^2d\mu_m<\varepsilon.$$
For this $\epsilon>0$, there exists a neighborhood $\mathcal{O}(m)$ of $m$ such that
$$\bigg|\int_{S_m}\big|p(1-au)-1\big|^2d\mu_m-\int_{S_m}\big|p(1-au)-1\big|^2\cdot\big|k_z\big|^2\frac{d\theta}{2\pi}\bigg|<\epsilon$$
for $z\in \mathcal{O}(m) \cap \mathbb D$.
Changing of variable gives
$$\int_{S_m}\big|p\circ\phi_z(1-au\circ\phi_z)-1\big|^2\frac{d\theta}{2\pi}<2\epsilon.$$
Applying the H\"{o}lder inequality, we obtain that
\begin{align*}
&\big\|(I-P)\big\{(g\circ \phi_z)\cdot [p\circ \phi_z(1-au\circ\phi_z)-1]\big\}\big\|_{4/3} \\
& ~~~\leqslant C_1 \|g\circ \phi_z\|_4\cdot\|p\circ \phi_z(1-au\circ \phi_z)-1\|_2\leqslant C_1\|g\|_\infty\epsilon^{\frac{1}{2}}
\end{align*}
for some constant $C_1>0$. Combining the above inequality with the identity
$$(I-P)\big\{(g\circ \phi_z)(p\circ \phi_z)(1-au\circ \phi_z)\big\}=S_{p\circ \phi_z}H_{g\circ \phi_z}(1-au\circ \phi_z)$$
gives us
\begin{align*}
\|(I-P)(g\circ\phi_z)\|_{4/3} & \leqslant C_1 \|g\|_{\infty} \epsilon^{\frac{1}{2}}+\big\|(I-P)\big\{(g\circ \phi_z)(p\circ \phi_z)(1-au\circ \phi_z)\big\}\big\|_{4/3}\\
 & \leqslant C_1 \|g\|_{\infty} \epsilon^{\frac{1}{2}}+\|p\|_{\infty}\cdot \big\|(I-P)[(1-au\circ\phi_z)(g\circ\phi_z)]\big\|_2.
\end{align*}
Recalling that
$$\lim\limits_{z\rightarrow m}\big\|(I-P)\big[(1-au\circ\phi_z)(g\circ\phi_z)\big]\big\|_2=0,$$
we get $$\lim\limits_{z\rightarrow m}\|(I-P)(g\circ\phi_z)\|_{4/3}\leqslant C_1 \|g\|_\infty\epsilon^{\frac{1}{2}}.$$

Note that the projection $P$ is bounded on $L^4$, there exists an absolute constant $C>0$  such that
\begin{align*}
\left\|(I-P)(g\circ\phi_z)\right\|_{4} \leqslant C \|g\|_\infty.
\end{align*}
In addition, since
\begin{align*}
\left\|(I-P)(g\circ\phi_z)\right\|_2^2 & \leqslant \left\|(I-P)(g\circ\phi_z)\right\|_{4/3} \cdot
\left\|(I-P)(g\circ\phi_z)\right\|_{4},
\end{align*}
it follows  that $$\lim\limits_{z\rightarrow m}\|H_{g}k_z\|_2=\lim\limits_{z\rightarrow m}\|(I-P)(g\circ\phi_z)\|_2=0.$$
Thus we conclude by Lemma \ref{2e} that $g|_{S_m}\in H^\infty|_{S_m}$, which contradicts our assumption.

\textbf{Case 3.} Suppose that $\ \mathrm{dim \ (span} \{[f|_{S_m}],\ [g|_{S_m}]\})=2$. In this case, we need to further consider the dimension of
$\ \mathrm{span} \{[\overline{f}|_{S_m}],\ [\overline{g}|_{S_m}]\}$.

\textbf{Subcase 3(i).} \ If $\ \mathrm{dim \ (span} \{[\overline{f}|_{S_m}],\ [\overline{g}|_{S_m}]\})=0$, then we have $\overline{f}|_{S_m},\ \overline{g}|_{S_m}\in H^{\infty}|_{S_m}$.

\textbf{Subcase 3(ii).} \ Suppose that $\ \mathrm{dim \ (span} \{[\overline{f}|_{S_m}],\ [\overline{g}|_{S_m}]\})=1$. Without loss of generality,
we may assume that $\left[\overline{g}|_{S_m}\right]\neq 0$ and $\left[\overline{f}|_{S_m}\right]=d\left[\overline{g}|_{S_m}\right]$ for some constant  $d$.
Then $\left(\overline{f}-d\overline{g}\right)|_{S_m}\in H^{\infty}|_{S_m}$, we have by Lemma \ref{2e}  that
$$H_{\overline{f}}k_z=dH_{\overline{g}}k_z+\varepsilon(z)\ \ \ \ \mathrm{and}\ \ \ \ H_{u\overline{f}}k_z=dH_{u\overline{g}}k_z+\varepsilon(z),$$
where the second equation follows from that $H_{u\varphi}=S_uH_{\varphi}$ for all $\varphi \in L^\infty$.  Thus we can rewrite (\ref{2.a}) as follows:
$$H_{\overline{g}}k_z\otimes H_{f-\overline{d}g}k_z=H_{u\overline{g}}k_z\otimes H_{u(f-\overline{d}g)}k_z+\varepsilon(z).$$
Using the same arguments as the one in Case 2, we conclude that $\left(f-\overline{d}g\right)|_{S_m}\in H^{\infty}|_{S_m}$.
So we have that $\left(f-\overline{d}g\right)|_{S_m}$ is a constant, as desired.

\textbf{Subcase 3(iii).} Finally, we consider the case that $\ \mathrm{dim \ (span} \{[\overline{f}|_{S_m}],\ [\overline{g}|_{S_m}]\})=2$. In this subcase,
$\varliminf\limits_{z\rightarrow m}\|H_{f}k_z\|_2$, $\varliminf\limits_{z\rightarrow m}\|H_{g}k_z\|_2$,
$\varliminf\limits_{z\rightarrow m}\|H_{\overline{f}}k_z\|_2$ and $\varliminf\limits_{z\rightarrow m}\|H_{\overline{g}}k_z\|_2$ are all positive.
By (\ref{2.a}), we have
\begin{equation}\label{2.aa}
\begin{array}{l}
\langle H_{u\overline{g}}k_z,H_{\overline{g}}k_z\rangle H_fk_z-\langle H_{u\overline{g}}k_z,H_{\overline{f}}k_z\rangle H_gk_z \vspace{2mm}\\
=\langle H_{u\overline{g}}k_z,H_{u\overline{g}}k_z\rangle H_{uf}k_z-\langle H_{u\overline{g}}k_z,H_{u\overline{f}}k_z\rangle H_{ug}k_z+\varepsilon(z)
\end{array}
\end{equation}
and
\begin{equation}\label{2.ab}
\begin{array}{l}
\langle H_{u\overline{f}}k_z,H_{\overline{g}}k_z\rangle H_fk_z-\langle H_{u\overline{f}}k_z,H_{\overline{f}}k_z\rangle H_gk_z \vspace{2mm}\\
=\langle H_{u\overline{f}}k_z,H_{u\overline{g}}k_z\rangle H_{uf}k_z-\langle H_{u\overline{f}}k_z,H_{u\overline{f}}k_z\rangle H_{ug}k_z+\varepsilon(z).
\end{array}
\end{equation}

In order to complete the discussion of Subcase 3(iii), the following claim is required.
\begin{cla}\label{2aa}
$\varliminf\limits_{z\rightarrow m}\left(\|H_{u\overline{f}}k_z\|_2^2\cdot\|H_{u\overline{g}}k_z\|_2^2-|\langle H_{u\overline{f}}k_z,H_{u\overline{g}}k_z\rangle|^2\right)=\delta>0$ for some $\delta$.
\end{cla}

As the proof of the above claim is long, let us assume this for the moment and we will give its proof later.

Based on Claim \ref{2aa}, we have by (\ref{2.aa}) and (\ref{2.ab}) that there are  $a_{11}(z)$, $a_{12}(z)$, $a_{21}(z)$ and $a_{22}(z)$ such that
\begin{equation} \label{2.e}
\begin{cases}
H_{uf}k_z = a_{11}(z)H_{f}k_z+a_{12}(z)H_{g}k_z+\varepsilon(z),\vspace{2mm}\\
H_{ug}k_z = a_{21}(z)H_{f}k_z+a_{22}(z)H_{g}k_z+\varepsilon(z),\vspace{2mm}\\
\end{cases}
\end{equation}
where $z\in\mathcal{O}(m)\cap \mathbb D$. Furthermore,  observe that the functions $\left\{a_{ij}(z)\right\}_{i,j=1}^2$ are all uniformly bounded for $z\in\mathcal{O}(m)\cap \mathbb D$.

Using the Bolzano-Weierstrass theorem again, there exist constants  $\big\{a_{11},a_{12},a_{21},a_{22}\big\}$ which are independent of $z$  such that for $z\in \mathcal{O}(m)\cap \mathbb D$:
\begin{equation} \label{2.i}
\begin{cases}
H_{uf}k_z = a_{11}H_{f}k_z+a_{12}H_{g}k_z+\varepsilon(z),\vspace{2mm}\\
H_{ug}k_z = a_{21}H_{f}k_z+a_{22}H_{g}k_z+\varepsilon(z).\vspace{2mm}\\
\end{cases}
\end{equation}

Without loss of generality, we may assume that the coefficient matrix of (\ref{2.i}) has the following form:
\begin{gather*}
\begin{pmatrix} \lambda_1 & 1 \\ 0 & \lambda_1\end{pmatrix}
\ \ \ \ \  \mathrm{or}\ \ \ \ \   \begin{pmatrix} \lambda_1 & 0 \\ 0 & \lambda_2\end{pmatrix},
\end{gather*}
where the above two matrices are the Jordan canonical forms for $(a_{ij})$.   In fact, there is  an invertible matrix
\begin{gather*}
B=\begin{pmatrix}
b_{11} & b_{12}\\
b_{21} & b_{22}
\end{pmatrix}
\end{gather*}
such that
\begin{align*}
B\binom{H_{uf}k_z} {H_{ug}k_z}=\begin{pmatrix} \lambda_1 & 0\\
0 & \lambda_2\end{pmatrix}B \binom{H_{f}k_z}{H_{g}k_z}+\varepsilon(z)
\end{align*}
or
\begin{align*}
B\binom{H_{uf}k_z} {H_{ug}k_z}=\begin{pmatrix} \lambda_1 & 1\\
0 & \lambda_1\end{pmatrix}B \binom{H_{f}k_z}{H_{g}k_z}+\varepsilon(z).
\end{align*}
This gives that
\begin{align*}
\binom{H_{u(b_{11}f+b_{12}g)}k_z} {H_{u(b_{21}f+b_{22}g)}k_z}=\begin{pmatrix} \lambda_1 & 0\\
0 & \lambda_2\end{pmatrix} \binom{H_{(b_{11}f+b_{12}g)}k_z}{H_{(b_{21}f+b_{22}g)}k_z}+\varepsilon(z)
\end{align*}
or
\begin{align*}
\binom{H_{u(b_{11}f+b_{12}g)}k_z} {H_{u(b_{21}f+b_{22}g)}k_z}=\begin{pmatrix} \lambda_1 & 1\\
0 & \lambda_1\end{pmatrix} \binom{H_{(b_{11}f+b_{12}g)}k_z}{H_{(b_{21}f+b_{22}g)}k_z}+\varepsilon(z).
\end{align*}
Now define $F=b_{11}f+b_{12}g$  and $G=b_{21}f+b_{22}g$. Then  we have  that $\overline{f}|_{S_m},\ \overline{g}|_{S_m}\in H^\infty|_{S_m}$ if and only if $\overline{F}|_{S_m},\ \overline{G}|_{S_m}\in H^\infty|_{S_m}$, since the matrix $(b_{ij})$ is invertible.

If the above coefficient matrix  for   (\ref{2.i}) is
\begin{gather*}
\begin{pmatrix} \lambda_1 & 0 \\ 0 & \lambda_2\end{pmatrix},
\end{gather*}
then we have
\begin{eqnarray*}
\begin{cases}
H_{uf}k_z=\lambda_1 H_{f}k_z+\varepsilon(z),\vspace{2mm}\\
H_{ug}k_z=\lambda_2 H_{g}k_z+\varepsilon(z).\vspace{2mm}\\
\end{cases}
\end{eqnarray*}
Solving the above system gives
 $$|\lambda_1|=\frac{|\langle S_uH_{f}k_z,H_{f}k_z\rangle|}{\|H_{f}k_z\|_2^2}+\varepsilon(z)$$
 and
$$|\lambda_2|=\frac{|\langle S_uH_{g}k_z,H_{g}k_z\rangle|}{\|H_{g}k_z\|_2^2}+\varepsilon(z)$$ for all  $z\in \mathcal{O}(m)\cap \mathbb D$.
Since $u$ is an inner function, we conclude that $|\lambda_1|\leqslant 1$ and $|\lambda_2|\leqslant 1$. Thus we have
\begin{eqnarray*}
H_{\overline{g}}k_z\otimes H_{f}k_z- H_{\overline{f}}k_z\otimes H_{g}k_z=H_{u\overline{g}}k_z\otimes \lambda_1H_{f}k_z- H_{u\overline{f}}k_z\otimes \lambda_2H_{g}k_z+\varepsilon(z),
\end{eqnarray*}
to obtain
$$H_{(1-\overline{\lambda}_1u)\overline{g}}k_z\otimes H_{f}k_z=H_{(1-\overline{\lambda}_2u)\overline{f}}k_z\otimes H_{g}k_z+\varepsilon(z)$$
for $z\in \mathcal{O}(m)\cap \mathbb D$.
This gives that
\begin{eqnarray}\label{2.ai}
\begin{cases}
\langle H_{f}k_z,H_{f}k_z\rangle H_{(1-\overline{\lambda}_1u)\overline{g}}k_z=\langle H_{f}k_z,H_{g}k_z\rangle H_{(1-\overline{\lambda}_2u)\overline{f}}k_z +\varepsilon(z),\vspace{2mm}\\
\langle H_{g}k_z,H_{f}k_z\rangle H_{(1-\overline{\lambda}_1u)\overline{g}}k_z=\langle H_{g}k_z,H_{g}k_z\rangle H_{(1-\overline{\lambda}_2u)\overline{f}}k_z +\varepsilon(z),
\end{cases}
\end{eqnarray}
where $z\in \mathcal{O}(m)\cap \mathbb D$.

Since $\left[f|_{S_m}\right]$ and $\left[g|_{S_m}\right]$ are linearly independent,
we first show that
\begin{align}\label{C-Z}
\varliminf\limits_{z\rightarrow m}\left(\|H_{f}k_z\|_2^2\cdot\|H_{g}k_z\|_2^2-|\langle H_{f}k_z,H_{g}k_z\rangle|^2\right)=\mu>0.
\end{align}
Otherwise, there is a net $\{z_{\beta}\}\subset \mathbb D$ such that
$$\lim\limits_{z_{\beta}\rightarrow m}\left(\|H_{f}k_{z_\beta}\|_2^2\cdot\|H_{g}k_{z_\beta}\|_2^2-|\langle H_{f}k_{z_\beta},H_{g}k_{z_\beta}\rangle|^2\right)=0.$$
For $z\in\mathcal{O}(m)\cap \mathbb D$, we let $$\lambda_{z}=\frac{\langle H_{f}k_{z},H_{g}k_{z}\rangle}{\|H_{g}k_{z}\|_2^2}.$$ Clearly, $\lambda_{z}$ is
uniformly bounded for $z\in\mathcal{O}(m)\cap \mathbb D$, since $\varliminf\limits_{z\rightarrow m}\|H_{g}k_{z}\|_2>0.$ Then
$$\left\|H_{f}k_z-\lambda_zH_{g}k_z\right\|_2^2=\frac{\|H_{f}k_z\|_2^2\cdot \|H_{g}k_z\|_2^2-|\langle H_{f}k_z,H_{g}k_z\rangle|^2}{\|H_{g}k_z\|_2^2}$$ for each $z$ in the neighbourhood $\mathcal{O}(m)\cap \mathbb D$. On the other hand, we can choose a subnet $\{z_{\beta,\gamma}\}$ of $\{z_\beta\}$ such that
$\lim\limits_{z_{\beta,\gamma}\rightarrow m}\lambda_{z_{\beta,\gamma}}=\lambda$ for some $\lambda$, and we also have
$$\lim\limits_{z_{\beta,\gamma}\rightarrow m}\left\|H_{f}k_{z_{\beta,\gamma}}-\lambda H_{g}k_{z_{\beta,\gamma}}\right\|_2=0.$$
Now Lemma \ref{2e} gives  $$\lim\limits_{z\rightarrow m}\left\|H_{f}k_{z}-\lambda H_{g}k_{z}\right\|_2=0,$$
to obtain that $\left(f-\lambda g\right)|_{S_m}\in H^{\infty}|_{S_m}$, which is impossible since our assumption is
$$\mathrm{dim \ (span} \{[f|_{S_m}],\ [g|_{S_m}]\})=2.$$ The contradiction implies that $\mu>0$.

By  (\ref{2.ai}), we have
$$\left(\|H_{f}k_z\|_2^2\cdot\|H_{g}k_z\|_2^2-|\langle H_{f}k_z,H_{g}k_z\rangle|^2\right)H_{(1-\overline{\lambda}_1u)}k_z=\varepsilon(z)$$
and
$$\left(\|H_{f}k_z\|_2^2\cdot\|H_{g}k_z\|_2^2-|\langle H_{f}k_z,H_{g}k_z\rangle|^2\right)H_{(1-\overline{\lambda}_2u)}k_z+\varepsilon(z)=0.$$
Thus we conclude by (\ref{C-Z}) that
$$\lim_{z\rightarrow m}\|H_{(1-\overline{\lambda}_1u)\overline{g}}k_z\|_2=0 \ \ \ \ \ \ \mathrm{and}
 \ \ \ \ \ \ \lim_{z\rightarrow m}\|H_{(1-\overline{\lambda}_2u)\overline{f}}k_z\|_2=0.$$
Repeating the same arguments as used in Case 2, we have $\overline{f}|_{S_m},\ \overline{g}|_{S_m}\in H^\infty|_{S_m}$, which is a contradiction.

In order to finish the proof, it remains to consider the case that the coefficient matrix of (\ref{2.i}) is
\begin{gather*}
\begin{pmatrix} \lambda_1 & 1 \\ 0 & \lambda_1\end{pmatrix}.
\end{gather*}
In this case, we have that for $z\in \mathcal{O}(m)\cap \mathbb D$:
\begin{eqnarray*}
\begin{cases}
H_{uf}k_z=\lambda_1 H_{f}k_z+H_{g}k_z+\varepsilon(z),\vspace{2mm}\\
H_{ug}k_z=\lambda_1 H_{g}k_z+\varepsilon(z).\vspace{2mm}\\
\end{cases}
\end{eqnarray*}
Using the same arguments as the above, we also have $|\lambda_1|\leqslant  1$ and
\begin{align*}
&H_{\overline{g}}k_z\otimes H_{f}k_z- H_{\overline{f}}k_z\otimes H_{g}k_z\\
&=H_{u\overline{g}}k_z\otimes (\lambda_1H_{f}k_z+H_{g}k_z)- H_{u\overline{f}}k_z\otimes \lambda_1H_{g}k_z+\varepsilon(z),
\end{align*}
which is equivalent to
$$H_{(1-\overline{\lambda}_1u)\overline{g}}k_z\otimes H_{f}k_z=H_{u\overline{g}+(1-\overline{\lambda}_1u)\overline{f}}k_z\otimes H_{g}k_z+\varepsilon(z).$$

Since $\left[f|_{S_m}\right]$ and $\left[g|_{S_m}\right]$ are linearly independent, we deduce that $H_{f}k_z$ and $H_{g}k_z$ are linearly independent
for all  $z\in \mathcal{O}(m)\cap \mathbb D$ and moreover,
\begin{equation} \label{2.j}
\begin{cases}
\big((1-\overline{\lambda}_1u)\overline{g}\big)|_{S_m}\in H^\infty|_{S_m},\vspace{2mm}\\
\big((1-\overline{\lambda}_1u)\overline{f}+(u\overline{g})\big)|_{S_m}\in H^\infty|_{S_m}.
\end{cases}
\end{equation}
Thus we obtain $$\overline{f}|_{S_m},\ \overline{g}|_{S_m}\in H^\infty|_{S_m}.$$
This contradicts our assumption that $\ \mathrm{dim \ (span} \{[\overline{f}|_{S_m}],\ [\overline{g}|_{S_m}]\})=2$.

To complete the whole proof of Proposition \ref{2c}, we need to show the following  result holds under the assumption that
$\varliminf\limits_{z\rightarrow m}\|H_{f}k_z\|_2$, $\varliminf\limits_{z\rightarrow m}\|H_{g}k_z\|_2$,
$\varliminf\limits_{z\rightarrow m}\|H_{\overline{f}}k_z\|_2$ and $\varliminf\limits_{z\rightarrow m}\|H_{\overline{g}}k_z\|_2$ are all positive:
$$\varliminf\limits_{z\rightarrow m}\left(\|H_{u\overline{f}}k_z\|_2^2\cdot\|H_{u\overline{g}}k_z\|_2^2-|\langle H_{u\overline{f}}k_z,H_{u\overline{g}}k_z\rangle|^2\right)=\delta>0.$$
\emph{\textbf{Proof of Claim \ref{2aa}.}} By the Cauchy-Schwarz inequality, we have
$$\varliminf\limits_{z\rightarrow m}\left(\|H_{u\overline{f}}k_z\|_2^2\cdot\|H_{u\overline{g}}k_z\|_2^2-|\langle H_{u\overline{f}}k_z,H_{u\overline{g}}k_z\rangle|^2\right)\geqslant 0.$$
If the above conclusion does not hold, we can find a net $\left\{z_\alpha\right\}\subset \D$ such that
$z_{\alpha}\rightarrow m$ and
$$\lim\limits_{z_{\alpha}\rightarrow m}\left(\|H_{u\overline{f}}k_{z_\alpha}\|_2^2\cdot\|H_{u\overline{g}}k_{z_\alpha}\|_2^2-|\langle H_{u\overline{f}}k_{z_\alpha},H_{u\overline{g}}k_{z_\alpha}\rangle|^2\right)=0.$$
We first show that $\varliminf\limits_{z\rightarrow m}\|H_{u\overline{g}}k_z\|_2>0$. If this is not the case, Lemma \ref{2e} gives
$$\lim\limits_{z\rightarrow m}\|H_{u\overline{g}}k_z\|_2=0.$$
Thus we can rewrite (\ref{2.a}) as follows:
\begin{equation}\label{2.af}
\begin{array}{l}
H_{\overline{f}}k_z\otimes H_gk_z-H_{\overline{g}}k_z\otimes H_fk_z
=H_{u\overline{f}}k_z\otimes H_{ug}k_z+\varepsilon(z).
\end{array}
\end{equation}
This implies that
\begin{equation}\label{2.ad}
\begin{array}{l}
\langle H_{\overline{g}}k_z,H_{\overline{g}}k_z\rangle H_fk_z-\langle H_{\overline{g}}k_z,H_{\overline{f}}k_z\rangle H_gk_z \vspace{2mm}\\
=-\langle H_{\overline{g}}k_z,H_{u\overline{f}}k_z\rangle H_{ug}k_z+\varepsilon(z)
\end{array}
\end{equation}
and
\begin{equation}\label{2.ae}
\begin{array}{l}
\langle H_{\overline{f}}k_z,H_{\overline{g}}k_z\rangle H_fk_z-\langle H_{\overline{f}}k_z,H_{\overline{f}}k_z\rangle H_gk_z \vspace{2mm}\\
=-\langle H_{\overline{f}}k_z,H_{u\overline{f}}k_z\rangle H_{ug}k_z+\varepsilon(z).
\end{array}
\end{equation}
Using the method as the one in the proof of (\ref{C-Z}), we obtain
\begin{align*}\label{2.ag}
\varliminf\limits_{z\rightarrow m}\left(\|H_{\overline{f}}k_z\|_2^2\cdot\|H_{\overline{g}}k_z\|_2^2-|\langle H_{\overline{f}}k_z,H_{\overline{g}}k_z\rangle|^2\right)>0,
\end{align*}
since $\left[\overline{f}|_{S_m}\right]$ and $\left[\overline{g}|_{S_m}\right]$ are also linearly independent. Therefore, we have by (\ref{2.ad}) and (\ref{2.ae}) that there exists $b(z)$ such that $$H_{g}k_z=b(z)H_{ug}k_z+\varepsilon(z)$$ for all $z\in\mathcal{O}(m)\cap \mathbb D$. Moreover,
$b(z)$ is uniformly bounded for $z\in\mathcal{O}(m)\cap \mathbb D$.
Thus we can choose a net $\{z_\zeta\}$ such that $\lim\limits_{z_\zeta\rightarrow m}b(z_\zeta)=b$. Using Lemma \ref{2e} again, we obtain that
\begin{equation}\label{H}
H_{g}k_z=bH_{ug}k_z+\varepsilon(z)
\end{equation}
for all $z\in\mathcal{O}(m)\cap \mathbb D$. As $\varliminf\limits_{z\rightarrow m}\|H_{g}k_z\|_2>0$ and
$$\|H_{g}k_z\|_2=\|bH_{ug}k_z+\varepsilon(z)\|_2\leqslant |b|\cdot\|H_{g}k_z\|_2+\|\varepsilon(z)\|_2,$$
we conclude that $|b|\geqslant 1$.

Using (\ref{2.af}), we have
$$H_{\overline{g}}k_z\otimes H_fk_z
=H_{(\overline{b}-u)\overline{f}}k_z\otimes H_{ug}k_z+\varepsilon(z)$$
and
$$H_{f}k_z=c(z)H_{ug}k_z+\varepsilon(z),$$
where $$c(z)=\frac{\langle H_{\overline{g}}k_z,H_{(\overline{b}-u)\overline{f}}k_z\rangle}{\|H_{\overline{g}}k_z\|_2^2}$$ is uniformly
bounded for $z\in\mathcal{O}(m)\cap \mathbb D$. So there is  a constant $c$ (which is independent of $z$) such that
\begin{equation*}
H_{f}k_z=cH_{ug}k_z+\varepsilon(z).
\end{equation*}
As we have shown
$$H_{g}k_z=bH_{ug}k_z+\varepsilon(z)$$ for all $z\in\mathcal{O}(m)\cap \mathbb D$, it follows that
 $$H_fk_z=\frac{c}{b}H_gk_z+\varepsilon(z).$$  This implies
$\left(f-\frac{c}{b}g\right)|_{S_m}\in H^\infty|_{S_m}$. But this contradicts  our assumption that $$\mathrm{dim \ (span} \{[f|_{S_m}],\ [g|_{S_m}]\})=2.$$
So we have $\varliminf\limits_{z\rightarrow m}\|H_{u\overline{g}}k_z\|_2>0$.

Recall that our assumption is
$$\varliminf\limits_{z\rightarrow m}\left(\|H_{u\overline{f}}k_z\|_2^2\cdot\|H_{u\overline{g}}k_z\|_2^2-|\langle H_{u\overline{f}}k_z,H_{u\overline{g}}k_z\rangle|^2\right)=0.$$
Using the same method  as the one in the proof of (\ref{C-Z}), there exists a constant $\lambda'$ such that
$$\lim\limits_{z\rightarrow m} \|H_{u\overline{f}}k_z-\lambda'H_{u\overline{g}}k_z\|_2=0.$$
Combining the above limit with  (\ref{2.a}) gives that
\begin{equation*}
\begin{array}{l}
H_{\overline{g}}k_z\otimes H_fk_z-H_{\overline{f}}k_z\otimes H_gk_z
=H_{u\overline{g}}k_z\otimes H_{u(f-\overline{\lambda'}g)}k_z+\varepsilon(z).
\end{array}
\end{equation*}
Rewrite the above formula as the following:
\begin{align}\label{2.ah}
H_{\overline{g}}k_z\otimes H_{(f-\overline{\lambda'}g)}k_z-H_{(\overline{f}-\lambda'\overline{g})}k_z\otimes H_gk_z
=H_{u\overline{g}}k_z\otimes H_{u(f-\overline{\lambda'}g)}k_z+\varepsilon(z).
\end{align}
Since
$$ \mathrm{dim \ (span} \{[f|_{S_m}],\ [g|_{S_m}]\})= \mathrm{dim \ (span} \{[\overline{f}|_{S_m}],\ [\overline{g}|_{S_m}]\})=2,$$
we have
\begin{align*}
 \mathrm{dim \ \big(span} \big\{[(f-\overline{\lambda'}g)|_{S_m}],\ [g|_{S_m}]\big\}\big)&=\mathrm{dim \ \big(span} \big\{[(\overline{f}-\lambda'\overline{g})|_{S_m}],\ [\overline{g}|_{S_m}]\big\}\big)=2.
\end{align*}

Comparing (\ref{2.ah}) with (\ref{2.af}) and  then repeating  the same arguments as used in (\ref{H}), we have
$$H_{(f-\overline{\lambda'}g)}k_z=b'H_{u(f-\overline{\lambda'}g)}k_z+\varepsilon(z)\ \ \ \  \ \mathrm{and} \ \ \ \ \ H_gk_z=c'H_{u(f-\overline{\lambda'} g)}k_z+\varepsilon(z),$$
where $b',\ c'$ are independent of $z$ and moreover,  $|b'|\geqslant  1$ and $c'\neq 0$, since $\left[g|_{S_m}\right]\neq 0$.
Thus we have $$\lim\limits_{z\rightarrow m}\Big\|H_{\left(g-\frac{c'}{b'}(f-\overline{\lambda'}g)\right)}k_z\Big\|_2=0.$$  This yields that
$$\Big(g-\frac{c'}{b'}(f-\overline{\lambda'}g)\Big)\Big|_{S_m}\in H^\infty|_{S_m}.$$  But it is a contradiction, since $ \mathrm{dim \ (span} \{[f|_{S_m}],\ [g|_{S_m}]\})=2$. This completes the proof of Claim \ref{2aa} and hence the proof of Proposition \ref{2c}.
\end{proof}

Combining the preceding proposition with  the two  relations in Remark \ref{rem1},
we obtain the following proposition which gives a necessary condition for the compactness of the fourth operator $H_{uf}H_{u\overline{g}}^*+S_fS_g-H_{ug}H_{u\overline{f}}^*-S_gS_f$ in Lemma \ref{1b}.
\begin{prop}\label{2f}
Let $u$ be a nonconstant inner function, $f,g\in L^\infty$ and $m\in \mathcal M(H^\infty + C).$ Suppose that the operator $$H_{uf}H_{u\overline{g}}^*+S_fS_g-H_{ug}H_{u\overline{f}}^*-S_gS_f$$ is compact. Then for the support set $S_m$ of $m$, one of the following conditions holds:\\
  $(1)$ both $f|_{S_m}$ and $g|_{S_m}$ are in $H^\infty|_{S_m}$;\\
  $(2)$ both $\overline{f}|_{S_m}$ and $\overline{g}|_{S_m}$ are in $H^\infty|_{S_m}$;\\
  $(3)$ there exist constants $a$, $b$, not both zero, such that $(af+bg)|_{S_m}$ is a constant.
\end{prop}

Next, we will obtain a necessary condition for the  compactness of the second operator $T_fH_{u\overline{g}}^*+H_{u\overline{f}}^*S_g-T_gH_{u\overline{f}}^*-H_{u\overline{g}}^*S_f$ in Lemma \ref{1b}.
To do so, we need the following two lemmas.
\begin{lem}\label{2g}
Let $f$ and $g$ be in $L^2$. Then
$$H_fT_gT_{\phi_z}-S_{\phi_z}H_fT_g=H_fk_z\otimes T_{\overline{g\phi_z}}k_z=-(H_{f}k_z)\otimes(VH_{g}k_z)$$
for all $z\in \mathbb D$.
\end{lem}
\begin{proof}
Using  the identity (see \cite[Page 480]{Zheng}): $$I=k_z\otimes k_z+T_{\phi_z}T_{\overline{\phi_z}},$$
we obtain
\begin{align*}
H_fT_gT_{\phi_z} &=H_{f}(k_z\otimes k_z+T_{\phi_z}T_{\overline{\phi_z}})T_{g}T_{\phi_z}\\
                 &=(H_{f}k_z\otimes k_z)T_{g\phi_z}+ H_{f}T_{\phi_z}T_{\overline{\phi_z}g\phi_z}\\
                 &=(H_{f}k_z)\otimes (T_{\overline{g}\overline{\phi_z}}k_z)+H_{f}T_{\phi_z}T_{g}.
\end{align*}
Using Identity (4.6) of \cite{StZ1}, we have  $$S_{\phi_z}H_f=H_fT_{\phi_z}.$$ It follows that
$$H_fT_gT_{\phi_z}-S_{\phi_z}H_fT_g=(H_fk_z)\otimes (T_{\overline{g}\overline{\phi_z}}k_z).$$

To obtain the last equality, we recall that $V^2=I$ and  observe that
\begin{align*}
VT_{\overline{g\phi_z}}k_z &=VP\left(\overline{g\phi_z}k_z\right)\\
                                  &=(I-P)V\left(\overline{g\phi_z}k_z\right)\\
                                  &=(I-P)\left(\overline{w}g(w)\phi_z(w)\overline{k_z(w)}\right)\\
                                  &=(I-P)\left(\overline{w}g(w)\frac{z-w}{1-\overline{z}w}\frac{\sqrt{1-|z|^2}}{1-z\overline{w}}\right)\\
                                  &=-(I-P)\left(g(w)\frac{\sqrt{1-|z|^2}}{1-\overline{z}w}\right)\\
                                  &=-H_{g}k_z,
\end{align*}
which gives the desired result.
\end{proof}

\begin{lem}\label{2h}
Let $K:H^2\rightarrow \overline{zH^2}$ be a compact operator. Then
$$\lim_{|z|\rightarrow 1^-}\|S_{\phi_z}K-KT_{\phi_z}\|=0.$$
\end{lem}

\begin{proof}
Since each compact operator can be approximated by finite rank operator in norm, we need only to consider the case that $K$ is a rank-one operator.

Suppose that $K=f\otimes g$, where $f\in \overline{zH^2}$ and $g\in H^2$. Then
\begin{align*}
S_{\phi_z}K-KT_{\phi_z} &=S_{\phi_z}(f\otimes g)-(f\otimes g)T_{\phi_z}\\
                        &=(S_{\phi_z}f)\otimes g - f\otimes (T_{\phi_z}^*g).
\end{align*}
For every $w$ on $\partial \D$, let $|z|\rightarrow 1^-$, we have $$z-\phi_z(w)=\frac{1-|z|^2}{1-\overline{z}w}w\rightarrow 0.$$ So we have by the dominated convergence theorem that
$$\|zf-\phi_zf\|_2\rightarrow 0 \ \ \ \ \  \mathrm{and}\ \ \ \ \  \|\overline{z}g-\overline{\phi_z}g\|_2\rightarrow 0$$ as $|z|\rightarrow 1^-$. It follows that
$\|\xi f-\phi_zf\|_2\rightarrow 0$ and $\|\overline{\xi}g-\overline{\phi_z}g\|_2\rightarrow 0$ if $z\rightarrow \xi\in \partial \mathbb D$.

Using the assumption that $f\in \overline{zH^2}$ and $g\in H^2$, we obtain
$$\|\xi f-S_{\phi_z}f\|_2=\|\xi f-(I-P)(\phi_{z}f)\|_2\rightarrow 0$$ and
$$\|\overline{\xi} g-T_{\phi_z}^*g\|_2=\|\overline{\xi} g-P(\overline{\phi_z}g)\|_2\rightarrow 0$$ as $z\rightarrow \xi$.
Then we obtain that
\begin{align*}
\|(S_{\phi_z}f)\otimes g - f\otimes (T_{\phi_z}^*g)\|&= \|(S_{\phi_z}f)\otimes g - \xi f\otimes g + f\otimes \overline{\xi} g - f\otimes T_{\phi_z}^*g\| \\
                                                    &\leqslant \|(S_{\phi_z}f)\otimes g - \xi f\otimes g\| + \|f\otimes (\overline{\xi} g) - f\otimes (T_{\phi_z}^*g)\|\\
                                                    &=\|(S_{\phi_z}f-\xi f)\otimes g\|+ \|f\otimes (\overline{\xi}g - T_{\phi_z}^*g)\|\\
                                                    &= \|S_{\phi_z}f-\xi f\|_2 \cdot\|g\|_2 + \|f\|_2 \cdot \|\overline{\xi}g - T_{\phi_z}^*g\|_2,
\end{align*}
to get
$$\lim\limits_{|z|\rightarrow 1^-} \|S_{\phi_z}f\otimes g - f\otimes T_{\phi_z}^*g\|=0,$$
which completes the proof.
\end{proof}

\begin{rem}
The Carleson-Corona theorem (\cite{Gar}) tells us that the conclusions of Lemmas \ref{2b} and \ref{2h} are equivalent to the condition that for each $m\in \mathcal M(H^\infty+C)$,
$$\lim\limits_{z\rightarrow m}\|K-T_{\phi_z}^*KT_{\phi_z}\|=0\ \ \ \ \ \mathrm{and} \ \ \ \ \ \lim_{z\rightarrow m}\|S_{\phi_z}K-KT_{\phi_z}\|=0$$
for $z$ in the unit disk $\mathbb D$.
\end{rem}

Combining Lemmas \ref{2g} and \ref{2h},  we obtain the following necessary condition for the compactness of the operator $T_fH_{u\overline{g}}^*+H_{u\overline{f}}^*S_g-T_gH_{u\overline{f}}^*-H_{u\overline{g}}^*S_f$.

\begin{prop}\label{2d}
Suppose that $u$ is a nonconstant inner function and $f,g\in L^\infty$. Let $m\in \mathcal M(H^\infty + C)$ and  $S_m$ be its support set.
Suppose that the operator $$T_fH_{u\overline{g}}^*+H_{u\overline{f}}^*S_g-T_gH_{u\overline{f}}^*-H_{u\overline{g}}^*S_f$$ is compact, and $f|_{S_m},\ g|_{S_m}\in H^\infty|_{S_m}$. Then either\\
$(1)$ $\left((u-\lambda)\overline{f}\right)|_{S_m}$ and  $\left((u-\lambda)\overline{g}\right)|_{S_m}$ are in $H^\infty|_{S_m}$ for some constant $\lambda$; or\\
$(2)$ there exist constants $a$, $b$, not both zero, such that $(af+bg)|_{S_m}$ is a constant.
\end{prop}

\begin{proof}
Suppose that
\begin{eqnarray}\label{2.k}
T_fH_{u\overline{g}}^*+H_{u\overline{f}}^*S_g-T_gH_{u\overline{f}}^*-H_{u\overline{g}}^*S_f=K
\end{eqnarray}
for some compact operator $K$. Taking adjoint of (\ref{2.k}), we have
$$H_{u\overline{g}}T_{\overline{f}}+S_{\overline{g}}H_{u\overline{f}}-H_{u\overline{f}}T_{\overline{g}}-S_{\overline{f}}H_{u\overline{g}}=K^*.$$
By Identity (4.5) of \cite{StZ1}:
$$H_{\varphi \psi}=H_{\varphi}T_{\psi}+S_{\varphi}H_{\psi}=H_{\psi}T_{\varphi}+S_{\psi}H_{\varphi}$$ for any $\varphi,\ \psi\in L^\infty$,
we also have
$$H_{u\overline{g}}T_{\overline{f}}-H_{\overline{g}}T_{u\overline{f}}-H_{u\overline{f}}T_{\overline{g}}+H_{\overline{f}}T_{u\overline{g}}=K^*.$$
From Lemma \ref{2g}, we have
\begin{align*}
K^*T_{\phi_z}-S_{\phi_z}K^*&=H_{u\overline{g}}k_z\otimes T_{f\overline{\phi_z}}k_z-H_{\overline{g}}k_z\otimes T_{\overline{u}f\overline{\phi_z}}k_z\\
&~~~~~ \ \ \ \ -H_{u\overline{f}}k_z\otimes T_{g\overline{\phi_z}}k_z+H_{\overline{f}}k_z\otimes T_{\overline{u}g\overline{\phi_z}}k_z.
\end{align*}
By Lemma \ref{2h}, the norm of the left hand side in the above equality tends to $0$ as $z\rightarrow m$. Thus we obtain

\begin{equation}\label{2.l}
\begin{array}{l}
H_{u\overline{f}}k_z\otimes T_{g\overline{\phi_z}}k_z-H_{u\overline{g}}k_z\otimes T_{f\overline{\phi_z}}k_z\vspace{2mm}\\
=H_{\overline{f}}k_z\otimes T_{\overline{u}g\overline{\phi_z}}k_z-H_{\overline{g}}k_z\otimes T_{\overline{u}f\overline{\phi_z}}k_z+\varepsilon(z).
\end{array}
\end{equation}
By Lemma \ref{2g} and (\ref{2.l}), we have
\begin{equation}\label{2.m}
\begin{array}{l}
H_{u\overline{g}}k_z\otimes VH_{\overline{f}}k_z-H_{u\overline{f}}k_z\otimes VH_{\overline{g}}k_z \vspace{2mm}\\
=H_{\overline{g}}k_z\otimes VH_{u\overline{f}}k_z-H_{\overline{f}}k_z\otimes VH_{u\overline{g}}k_z+\varepsilon(z).
\end{array}
\end{equation}
For $[\overline{f}|_{S_m}],\ [\overline{g}|_{S_m}]\in (L^{\infty}|_{S_m})/ (H^{\infty}|_{S_m})$,
the dimension of $\mathrm{span}\left\{[\overline{f}|_{S_m}],\ [\overline{g}|_{S_m}]\right\}$ should be $0$, $1$, or $2$. Let us analyse these three cases in the following.

\textbf{Case 1.}\ If $\mathrm{dim}\left(\mathrm{span} \left\{[\overline{f}|_{S_m}],\ [\overline{g}|_{S_m}]\right\}\right)=0$, then $[\overline{f}|_{S_m}]=[\overline{g}|_{S_m}]=0$, which
implies that $\overline{f}|_{S_m},\ \overline{g}|_{S_m}\in H^{\infty}|_{S_m}$. This gives that $f|_{S_m}$ and $g|_{S_m}$ are constants, and $(f+g)|_{S_m}$ is also a constant.

\textbf{Case 2.}\ If $\mathrm{dim}\left(\mathrm{span} \left\{[\overline{f}|_{S_m}],\ [\overline{g}|_{S_m}]\right\}\right)=1$, we assume that $[\overline{g}|_{S_m}]\neq 0$. Then
there is a constant $\lambda$ such that $[(\overline{f}+\lambda \overline{g})|_{S_m}]=0$, i.e.,
$$(\overline{f}+\lambda \overline{g})|_{S_m}\in H^{\infty}|_{S_m}.$$
On the other hand, since $f|_{S_m},\ g|_{S_m}\in H^{\infty}|_{S_m}$, we get that $(f+\overline{\lambda}g)|_{S_m}$ is a constant.

\textbf{Case 3.}\ If $\mathrm{dim}\left(\mathrm{span}\left\{[\overline{f}|_{S_m}],\ [\overline{g}|_{S_m}]\right\}\right)=2$, Lemma \ref{2e} gives that
$$\varliminf\limits_{z\rightarrow m}\|H_{\overline{f}}k_z\|_2\geqslant  d_1>0 \ \ \ \ \ \mathrm{and} \ \ \ \  \ \varliminf\limits_{z\rightarrow m}\|H_{\overline{g}}k_z\|_2\geqslant d_2>0$$
for some constants $d_1$ and $d_2$.
By (\ref{2.m}), we have
\begin{equation}\label{2.ba}
\begin{array}{l}
\langle  VH_{\overline{f}}k_z, VH_{\overline{f}}k_z\rangle H_{u\overline{g}}k_z-\langle VH_{\overline{f}}k_z,VH_{\overline{g}}k_z\rangle H_{u\overline{f}}k_z \vspace{2mm}\\
=\langle VH_{\overline{f}}k_z, VH_{u\overline{f}}k_z\rangle H_{\overline{g}}k_z-\langle VH_{\overline{f}}k_z, VH_{u\overline{g}}k_z\rangle H_{\overline{f}}k_z+\varepsilon(z)
\end{array}
\end{equation}
and
\begin{equation}\label{2.ba}
\begin{array}{l}
\langle  VH_{\overline{g}}k_z, VH_{\overline{f}}k_z\rangle H_{u\overline{g}}k_z-\langle VH_{\overline{g}}k_z,VH_{\overline{g}}k_z\rangle H_{u\overline{f}}k_z \vspace{2mm}\\
=\langle VH_{\overline{g}}k_z, VH_{u\overline{f}}k_z\rangle H_{\overline{g}}k_z-\langle VH_{\overline{g}}k_z, VH_{u\overline{g}}k_z\rangle H_{\overline{f}}k_z+\varepsilon(z).
\end{array}
\end{equation}
Since $V$ is anti-unitary, we also have
\begin{equation}\label{2.bc}
\begin{array}{l}
\|H_{\overline{f}}k_z\|_2^2 H_{u\overline{g}}k_z-\langle H_{\overline{g}}k_z,H_{\overline{f}}k_z\rangle H_{u\overline{f}}k_z \vspace{2mm}\\
=\langle H_{u\overline{f}}k_z, H_{\overline{f}}k_z\rangle H_{\overline{g}}k_z-\langle H_{u\overline{g}}k_z, H_{\overline{f}}k_z\rangle H_{\overline{f}}k_z+\varepsilon(z)
\end{array}
\end{equation}
and
\begin{equation}\label{2.bd}
\begin{array}{l}
\langle H_{\overline{f}}k_z, H_{\overline{g}}k_z\rangle H_{u\overline{g}}k_z-\|H_{\overline{g}}k_z\|_2^2 H_{u\overline{f}}k_z \vspace{2mm}\\
=\langle H_{u\overline{f}}k_z, H_{\overline{g}}k_z\rangle H_{\overline{g}}k_z-\langle H_{u\overline{g}}k_z, H_{\overline{g}}k_z\rangle H_{\overline{f}}k_z+\varepsilon(z).
\end{array}
\end{equation}
Using the same arguments as the one in the proof of Claim \ref{2aa}, we conclude that
$$\varliminf\limits_{z\rightarrow m}\left(\|H_{\overline{f}}k_z\|_2^2\cdot \|H_{\overline{g}}k_z\|_2^2-|\langle H_{\overline{f}}k_z,H_{\overline{g}}k_z\rangle|^2\right)=\rho$$ for some constant $\rho>0$.
By (\ref{2.bc}) and (\ref{2.bd}),  we can find $\{a_{ij}(z)\}_{i,j=1}^2$  such that
\begin{align*}
\binom{H_{u\overline{f}}k_z} {H_{u\overline{g}}k_z}=\begin{pmatrix} a_{11}(z) & a_{12}(z)\\
a_{21}(z) & a_{22}(z)\end{pmatrix} \binom{H_{\overline{f}}k_z}{H_{\overline{g}}k_z}+\varepsilon(z)
\end{align*}
for $z\in \mathcal{O}(m)\cap \mathbb D$, where $\{a_{ij}(z)\}_{i,j=1}^2$  are  uniformly bounded for  $z$ in $\mathcal{O}(m)\cap \mathbb D$. By the Bolzano-Weierstrass theorem and Lemma \ref{2e}, there are constants $\left\{a_{ij}\right\}_{i,j=1}^2$ (independent of $z$) such that
\begin{align*}
\binom{H_{u\overline{f}}k_z} {H_{u\overline{g}}k_z}=\begin{pmatrix} a_{11} & a_{12}\\
a_{21} & a_{22}\end{pmatrix} \binom{H_{\overline{f}}k_z}{H_{\overline{g}}k_z}+\varepsilon(z)
\end{align*}
for $z\in \mathcal{O}(m)\cap \mathbb D$,
to obtain
\begin{equation} \label{2.n}
\left\{
\begin{aligned}
H_{u\overline{f}}k_z &=a_{11}H_{\overline{f}}k_z+a_{12}H_{\overline{g}}k_z+\varepsilon(z),\\
H_{u\overline{g}}k_z &=a_{21}H_{\overline{f}}k_z+a_{22}H_{\overline{g}}k_z+\varepsilon(z).\\
\end{aligned}
\right.
\end{equation}
Combining (\ref{2.m}) and (\ref{2.n}), we have
\begin{eqnarray*}
H_{\overline{g}}k_z\otimes V(a_{22}H_{\overline{f}}k_z-a_{12}H_{\overline{g}}k_z)-H_{\overline{f}}k_z\otimes V(-a_{21}H_{\overline{f}}k_z+a_{11}H_{\overline{g}}k_z)\\
= H_{\overline{g}}k_z\otimes VH_{u\overline{f}}k_z-H_{\overline{f}}k_z\otimes VH_{u\overline{g}}k_z+\varepsilon(z).
\end{eqnarray*}
Since $\left[\overline{f}|_{S_m}\right]$ and $\left[\overline{g}|_{S_m}\right]$ are linearly independent, we obtain
\begin{equation} \label{2.o}
\begin{cases}
H_{u\overline{f}}k_z = a_{22}H_{\overline{f}}k_z-a_{12}H_{\overline{g}}k_z+\varepsilon(z),\vspace{2mm}\\
H_{u\overline{g}}k_z = -a_{21}H_{\overline{f}}k_z+a_{11}H_{\overline{g}}k_z+\varepsilon(z),
\end{cases}
\end{equation}
where $z\in \mathcal{O}(m)\cap \mathbb D$. (\ref{2.n}) and (\ref{2.o}) imply that $a_{11}=a_{22}$ and $a_{12}=a_{21}=0$. Thus there is a constant $\lambda$ such that
\begin{equation} \label{2.p}
\begin{cases}
H_{u\overline{f}}k_z = \lambda H_{\overline{f}}k_z+\varepsilon(z),\vspace{2mm}\\
H_{u\overline{g}}k_z = \lambda H_{\overline{g}}k_z+\varepsilon(z).\vspace{2mm}\\
\end{cases}
\end{equation}
Therefore,
$$\lim_{z\rightarrow m} \|H_{(u-\lambda)\overline{f}}k_z\|_2=\|H_{(u-\lambda)\overline{g}}k_z\|_2=0,$$
which implies that $\left((u-\lambda)\overline{f}\right)|_{S_m},\ \left((u-\lambda)\overline{g}\right)|_{S_m}\in H^{\infty}|_{S_m}$, to complete the proof of Proposition \ref{2d}.
\end{proof}

Proposition \ref{2d} yields the following necessary condition for the compactness of the third operator $H_{uf}T_g+S_fH_{ug}-H_{ug}T_f-S_gH_{uf}$ given in Lemma \ref{1b}.
\begin{prop}\label{2k}
Let $u$ be a nonconstant inner function, $f,g\in L^\infty$, and $m\in \mathcal M(H^\infty + C).$ Suppose that $\overline{f}|_{S_m},\ \overline{g}|_{S_m}\in H^\infty|_{S_m}$ and the operator
 $$H_{uf}T_g+S_fH_{ug}-H_{ug}T_f-S_gH_{uf}$$ is compact. Then either\\
  $(1)$ $\left((u-\lambda)f\right)|_{S_m}$  and  $\left((u-\lambda)g\right)|_{S_m}$ are in $H^\infty|_{S_m}$ for some constant $\lambda$; or\\
  $(2)$ there exist constants $a$, $b$, not both zero, such that $(af+bg)|_{S_m}$ is a constant.
\end{prop}
Combining Propositions \ref{2c}, \ref{2f}, \ref{2d} and \ref{2k},  we obtain the following necessary condition for the compactness of the commutator of  $D_f$ and $D_g$.
\begin{thm}\label{necessary condition}
Let $u$ be a nonconstant inner function, $f,g\in L^\infty$ and $m\in \mathcal M(H^\infty + C).$ If  $[D_f,D_g]$ is compact, then for the support set $S_m$ of $m$, one of the following  holds:\\
$(1)$ $f|_{S_m}$, $g|_{S_m}$, $\left((u-\lambda)\overline{f}\right)|_{S_m}$ and $\left((u-\lambda)\overline{g}\right)|_{S_m}$ are in $H^\infty|_{S_m}$ for some constant $\lambda$; \\
$(2)$ $\overline{f}|_{S_m}$, $\overline{g}|_{S_m}$, $\left((u-\lambda)f\right)|_{S_m}$ and $\left((u-\lambda)g\right)|_{S_m}$ are in $H^\infty|_{S_m}$ for some constant $\lambda$; \\
$(3)$ there exist constants $a$, $b$, not both zero, such that $(af+bg)|_{S_m}$ is a constant.
\end{thm}

\section{\label{S2}The sufficient part of  Theorem \ref{MR}}
In this section, we  will complete the proof of the sufficient part of Theorem \ref{MR}. To do so, we need two  lemmas.
\begin{lem}\label{3a}
Let $f, g\in L^\infty$ and
\begin{align*}
F_z&=H_{\overline{g}}k_z\otimes VH_{u\overline{f}}k_z-H_{u\overline{g}}k_z\otimes VH_{\overline{f}}k_z\\
&~~~~~ \ \ \ \ -H_{\overline{f}}k_z\otimes VH_{u\overline{g}}k_z+H_{u\overline{f}}k_z\otimes VH_{\overline{g}}k_z,
\end{align*}
where $z\in \mathbb D$. For each support set $S$, suppose that $f$ and $g$ satisfy one of the following conditions:\\
  $(1)$ $f|_S$, $g|_S$, $\left((u-\lambda)\overline{f}\right)|_{S}$ and $\left((u-\lambda)\overline{g}\right)|_{S}$ are in $H^\infty|_{S}$ for some constant $\lambda$; \\
  $(2)$ $\overline{f}|_S$, $\overline{g}|_S$, $\left((u-\lambda)f\right)|_{S}$ and $\left((u-\lambda)g\right)|_{S}$ are in $H^\infty|_{S}$ for some constant  $\lambda$; \\
  $(3)$ there exist constants $a$, $b$, not both zero, such that $(af+bg)|_{S}$ is a constant.\\
Then we have $$\lim\limits_{|z|\rightarrow 1^{-}}\|F_z\|=0.$$
\end{lem}

\begin{proof}
For each $m\in \mathcal M(H^\infty+C)$, let $S_m$ be the support set of $m$. By the Carleson-Corona theorem, we need only to show
$$\lim\limits_{z\rightarrow m}\|F_z\|=0.$$

If $f$ and $g$ satisfy Condition (2), then  we have by  Lemma \ref{2e} that  $$\lim\limits_{z\rightarrow m}\|H_{\overline{f}}k_z\|_2=0\ \ \ \ \  \mathrm{and}\ \ \ \ \ \lim\limits_{z\rightarrow m}\|H_{\overline{g}}k_z\|_2=0.$$ It follows that $$\lim\limits_{z\rightarrow m}\|F_z\|=0.$$

Assume that Condition (1) holds for $f$ and $g$, i.e.,
$$f|_{S_m},\ g|_{S_m},\ \left((u-\lambda)\overline{f}\right)|_{S_m}\ \ \ \ \ \mathrm{and} \ \ \ \ \  \left((u-\lambda)\overline{g}\right)|_{S_m}\in H^{\infty}|_{S_m}.$$
According to Lemma \ref{2e}, we have $$\lim\limits_{z\rightarrow m}\|H_{f}k_z\|_2=\lim\limits_{z\rightarrow m}\|H_{g}k_z\|_2=0,$$
and moreover, $$\lim\limits_{z\rightarrow m}\|H_{(u-\lambda)\overline{f}}k_z\|_2=\lim\limits_{z\rightarrow m}\|H_{(u-\lambda)\overline{g}}k_z\|_2=0.$$
Since
\begin{align*}
F_z&=H_{\overline{g}}k_z\otimes VH_{[(u-\lambda)\overline{f}+\lambda\overline{f}]}k_z-H_{u\overline{g}}k_z\otimes VH_{\overline{f}}k_z\\
&~~~~~ \ \ \ \ -H_{\overline{f}}k_z\otimes VH_{[(u-\lambda)\overline{g}+\lambda\overline{g}]}k_z+H_{u\overline{f}}k_z\otimes VH_{\overline{g}}k_z\\
&=H_{\overline{g}}k_z\otimes VH_{(u-\lambda)\overline{f}}k_z+\lambda H_{\overline{g}}k_z\otimes VH_{\overline{f}}k_z-H_{u\overline{g}}k_z\otimes VH_{\overline{f}}k_z\\
&~~~~~ \ \ \ \ -H_{\overline{f}}k_z\otimes VH_{(u-\lambda)\overline{g}}k_z-\lambda H_{\overline{f}}k_z\otimes VH_{\overline{g}}k_z+H_{u\overline{f}}k_z\otimes VH_{\overline{g}}k_z\\
&=H_{\overline{g}}k_z\otimes VH_{(u-\lambda)\overline{f}}k_z+H_{(\lambda-u)\overline{g}}k_z\otimes VH_{\overline{f}}k_z\\
&~~~~~ \ \ \ \ -H_{\overline{f}}k_z\otimes VH_{(u-\lambda)\overline{g}}k_z-H_{(\lambda-u)\overline{f}}k_z\otimes VH_{\overline{g}}k_z,
\end{align*}
we have
\begin{align*}
\|F_z\| &\leqslant \|H_{\overline{g}}k_z\otimes VH_{(u-\lambda)\overline{f}}k_z\|+\|H_{(\lambda-u)\overline{g}}k_z\otimes VH_{\overline{f}}k_z\| \\
&~~~~~ \ \ \ \ +\|H_{\overline{f}}k_z\otimes VH_{(u-\lambda)\overline{g}}k_z\|+\|H_{(\lambda-u)\overline{f}}k_z\otimes VH_{\overline{g}}k_z\| \\
&=\|H_{\overline{g}}k_z\|_2\cdot\|VH_{(u-\lambda)\overline{f}}k_z\|_2+\|H_{(\lambda-u)\overline{g}}k_z\|_2\cdot\|VH_{\overline{f}}k_z\|_2\\
&~~~~~ \ \ \ \ +\|H_{\overline{f}}k_z\|_2\cdot\|VH_{(u-\lambda)\overline{g}}k_z\|_2+\|H_{(\lambda-u)\overline{g}}k_z\|_2\cdot\|VH_{\overline{f}}k_z\|_2\\
&= \|H_{\overline{g}}k_z\|_2\cdot\|H_{(u-\lambda)\overline{f}}k_z\|_2+\|H_{(\lambda-u)\overline{g}}k_z\|_2\cdot\|H_{\overline{f}}k_z\|_2\\
&~~~~~ \ \ \ \ +\|H_{\overline{f}}k_z\|_2\cdot\|H_{(u-\lambda)\overline{g}}k_z\|_2+\|H_{(\lambda-u)\overline{g}}k_z\|_2\cdot\|H_{\overline{f}}k_z\|_2.
\end{align*}
This gives us that $$\lim\limits_{z\rightarrow m}\|F_z\|=0.$$

To finish our proof, we suppose that $f$ and $g$ satisfy Condition (3).
Without loss of generality, we may assume that $(f-ag)|_{S_m}=c$ for some constant $c$. Then we get that
$$(f-ag)|_{S_m},\ \ \ (\overline{f}-\overline{a}\overline{g})|_{S_m}\in H^{\infty}|_{S_m}$$ and
$$\left(u(f-ag)\right)|_{S_m},\ \ \ \left(u(\overline{f}-\overline{a}\overline{g})\right)|_{S_m}\in H^{\infty}|_{S_m}.$$
Noting that
\begin{align*}
F_z&=H_{\overline{g}}k_z\otimes VH_{[u(\overline{f}-\overline{a}\overline{g})+u\overline{a}\overline{g}]}k_z-H_{u\overline{g}}k_z\otimes VH_{(\overline{f}-\overline{a}\overline{g}+\overline{a}\overline{g})}k_z\\
&~~~~~ \ \ \ \ -H_{\overline{f}}k_z\otimes VH_{u\overline{g}}k_z+H_{u\overline{f}}k_z\otimes VH_{\overline{g}}k_z\\
&=H_{\overline{g}}k_z\otimes VH_{u(\overline{f}-\overline{a}\overline{g})}k_z+H_{\overline{a}\overline{g}}k_z\otimes VH_{u\overline{g}}k_z\\
&~~~~~ \ \ \ \ -H_{u\overline{g}}k_z\otimes VH_{(\overline{f}-\overline{a}\overline{g})}k_z-H_{\overline{a}u\overline{g}}k_z\otimes VH_{\overline{g}}k_z\\
&~~~~~ \ \ \ \ -H_{\overline{f}}k_z\otimes VH_{u\overline{g}}k_z+H_{u\overline{f}}k_z\otimes VH_{\overline{g}}k_z\\
&=H_{\overline{g}}k_z\otimes VH_{u(\overline{f}-\overline{a}\overline{g})}k_z+H_{(\overline{a}\overline{g}-\overline{f})}k_z\otimes VH_{u\overline{g}}k_z\\
&~~~~~ \ \ \ \ -H_{u\overline{g}}k_z\otimes VH_{(\overline{f}-\overline{a}\overline{g})}k_z-H_{u(\overline{a}\overline{g}-\overline{f})}k_z\otimes VH_{\overline{g}}k_z,
\end{align*}
we obtain
\begin{align*}
\|F_z\|&\leqslant\|H_{\overline{g}}k_z\otimes VH_{u(\overline{f}-\overline{a}\overline{g})}k_z\|+\|H_{(\overline{a}\overline{g}-\overline{f})}k_z\otimes VH_{u\overline{g}}k_z\|\\
&~~~~~ \ \ \ \ +\|H_{u\overline{g}}k_z\otimes VH_{(\overline{f}-\overline{a}\overline{g})}k_z\|+\|H_{u(\overline{a}\overline{g}-\overline{f})}k_z\otimes VH_{\overline{g}}k_z\|\\
&= 2\Big(\|H_{\overline{g}}k_z\|_2\cdot\|VH_{u(\overline{f}-\overline{a}\overline{g})}k_z\|_2+\|H_{u\overline{g}}k_z\|_2\cdot\|VH_{(\overline{f}-\overline{a}\overline{g})}k_z\|_2\Big)\\
&= 2\Big(\|H_{\overline{g}}k_z\|_2\cdot\|H_{u(\overline{f}-\overline{a}\overline{g})}k_z\|_2+\|H_{u\overline{g}}k_z\|_2\cdot\|H_{(\overline{f}-\overline{a}\overline{g})}k_z\|_2\Big).
\end{align*}
By Lemma \ref{2e} again, now we conclude that $$\lim\limits_{z\rightarrow m}\|F_z\|=0,$$ to complete the proof of Lemma \ref{3a}.
\end{proof}

The following lemma will be needed in the proof of Theorem \ref{MR}, which was established in \cite[Lemma 17]{GkZ1}.
\begin{lem}\label{3.d}
Suppose that $\varphi$ and $\psi$ are in $L^\infty$. Let $m\in \mathcal M(H^\infty+C)$. If
$$\lim\limits_{z\rightarrow m}\|H_{\varphi}k_z\|_2=0,$$
then we have
$$\lim\limits_{z\rightarrow m}\|H_{\varphi}T_{\psi}k_z\|_2=0.$$
\end{lem}

Now we are ready to complete the proof of Theorem \ref{MR}.
\begin{proof}[\emph{\textbf{\emph{Proof of the sufficient part of Theorem \ref{MR}.}}}] Let $m$ be in $\mathcal M(H^\infty+C)$. We
suppose that one of Conditions (1), (2) and (3) in Theorem \ref{MR} holds on the support set $S_m$. By Lemma \ref{1b}, we need to show that
$$T_fT_g+H_{u\overline{f}}^*H_{ug}-T_gT_f-H_{u\overline{g}}^*H_{uf},$$
$$T_fH_{u\overline{g}}^*+H_{u\overline{f}}^*S_g-T_gH_{u\overline{f}}^*-H_{u\overline{g}}^*S_f,$$
$$H_{uf}T_g+S_fH_{ug}-H_{ug}T_f-S_gH_{uf}$$
 and
$$H_{uf}H_{u\overline{g}}^*+S_fS_g-H_{ug}H_{u\overline{f}}^*-S_gS_f$$
are   compact.

Letting
\begin{align*}
K_1&=T_fT_g+H_{u\overline{f}}^*H_{ug}-T_gT_f-H_{u\overline{g}}^*H_{uf}\\
   &=(T_{fg}-H_{\overline{f}}^*H_g)+H_{u\overline{f}}^*H_{ug}-(T_{fg}-H_{\overline{g}}^*H_f)-H_{u\overline{g}}^*H_{uf}\\
   &=(H_{\overline{g}}^*H_f-H_{\overline{f}}^*H_g)-(H_{u\overline{g}}^*H_{uf}-H_{u\overline{f}}^*H_{ug}),
\end{align*}
we are going to show that $K_1$ is compact first.

In order to show that $K_1$ is compact, we first check that each condition of  Theorem \ref{MR} can imply (\ref{2.a}).
Indeed, if  Condition (1)  holds, then we have $$f|_{S_m},\ \ g|_{S_m}\in H^{\infty}|_{S_m}.$$
Using Lemma \ref{2e} and Identity (4.5) of \cite{StZ1}: $$H_{uf}=S_uH_f, \ \ \ \ H_{ug}=S_uH_g,$$
 we have $$\lim\limits_{z\rightarrow m}\|H_{f}k_z\|_2=\lim\limits_{z\rightarrow m}\|H_{g}k_z\|_2=0$$
and $$\lim\limits_{z\rightarrow m}\|H_{uf}k_z\|_2=\lim\limits_{z\rightarrow m}\|H_{ug}k_z\|_2=0,$$
which implies that
$$\lim\limits_{z\rightarrow m}\|H_{\overline{g}}k_z\otimes H_fk_z-H_{\overline{f}}k_z\otimes H_gk_z\|=0$$
and
$$\lim\limits_{z\rightarrow m}\|H_{u\overline{g}}k_z\otimes H_{uf}k_z-H_{u\overline{f}}k_z\otimes H_{ug}k_z\|=0.$$

Similarly, if $\overline{f}|_{S_m},\ \overline{g}|_{S_m}\in H^{\infty}|_{S_m}$, then we also have
$$\lim\limits_{z\rightarrow m}\|H_{\overline{g}}k_z\otimes H_fk_z-H_{\overline{f}}k_z\otimes H_gk_z\|=0$$
and
$$\lim\limits_{z\rightarrow m}\|H_{u\overline{g}}k_z\otimes H_{uf}k_z-H_{u\overline{f}}k_z\otimes H_{ug}k_z\|=0.$$
Thus Condition (1) or (2) in Theorem \ref{MR} can imply (\ref{2.a}).

If Condition (3)   holds, we have
$$(af+bg)|_{S_m}=c$$ for some constants $a,\ b,\ c$ with $|a|+|b|\neq 0$. Without loss of generality, we may assume that
$$(f-dg)|_{S_m}=e$$ for some constants $d$ and $e$. Then we have
\begin{align*}
H_{\overline{g}}k_z\otimes H_{f}k_z-H_{\overline{f}}k_z\otimes H_{g}k_z
&=H_{\overline{g}}k_z\otimes H_{(f-dg+dg)}k_z-H_{\overline{f}}k_z\otimes H_{g}k_z\\
&=H_{\overline{g}}k_z\otimes H_{(f-dg)}k_z+ \overline{d}H_{\overline{g}}k_z\otimes H_{g}k_z-H_{\overline{f}}k_z\otimes H_{g}k_z\\
&=H_{\overline{g}}k_z\otimes H_{(f-dg)}k_z+ H_{(\overline{dg}-\overline{f})}k_z\otimes H_{g}k_z
\end{align*}
and
\begin{align*}
&H_{u\overline{g}}k_z\otimes H_{uf}k_z-H_{u\overline{f}}k_z\otimes H_{ug}k_z \\
&=H_{u\overline{g}}k_z\otimes H_{u(f-dg+dg)}k_z-H_{u\overline{f}}k_z\otimes H_{ug}k_z\\
&=H_{u\overline{g}}k_z\otimes H_{u(f-dg)}k_z+ \overline{d}H_{u\overline{g}}k_z\otimes H_{ug}k_z -H_{u\overline{f}}k_z\otimes H_{ug}k_z\\
&= H_{u\overline{g}}k_z\otimes H_{u(f-dg)}k_z+ H_{u(\overline{dg}-\overline{f})}k_z\otimes H_{ug}k_z.
\end{align*}

Since $(f-dg)|_{S_m}$ is a constant, we conclude that $\big(u(f-dg)\big)|_{S_m}$ and $\big(u(\overline{dg}-\overline{f})\big)|_{S_m}$
both belong to $H^{\infty}|_{S_m}$.
Using Lemma \ref{2e} again, we get that
$$\lim\limits_{z\rightarrow m} \|H_{\overline{g}}k_z\otimes H_{f}k_z-H_{\overline{f}}k_z\otimes H_{g}k_z\|=0$$ and
$$\lim\limits_{z\rightarrow m} \|H_{u\overline{g}}k_z\otimes H_{uf}k_z-H_{u\overline{f}}k_z\otimes H_{ug}k_z\|=0,$$
which implies that the equation in (\ref{2.a}) holds, as desired.

By the definition of $K_1$ and Lemma \ref{2a}, we have
\begin{align*}
K_{1}-T_{\phi_{z}}^*K_{1}T_{\phi_{z}}&=\Big[(H_{\overline{g}}^*H_f-H_{\overline{f}}^*H_g)-(H_{u\overline{g}}^*H_{uf}-H_{u\overline{f}}^*H_{ug})\Big]\\
& \ \ \  \ \ -T_{\phi_{z}}^*\Big[(H_{\overline{g}}^*H_f-H_{\overline{f}}^*H_g)-(H_{u\overline{g}}^*H_{uf}-H_{u\overline{f}}^*H_{ug})\Big]T_{\phi_{z}}\\
&=V\Big[(H_{\overline{g}}k_z\otimes H_fk_z-H_{\overline{f}}k_z\otimes H_gk_z)\Big]V^* \\
& \ \ \  \ \ -V\Big[(H_{u\overline{g}}k_z\otimes H_{uf}k_z-H_{u\overline{f}}k_z\otimes H_{ug}k_z)\Big]V^*.
\end{align*}
It follows that
\begin{align}\label{compact1}
\lim\limits_{|z|\rightarrow 1^-}\|K_{1}-T_{\phi_{z}}^*K_{1}T_{\phi_{z}}\|=0.
\end{align}
On the other hand, since
$$H_{u\overline{f}}^*H_{ug}=T_{fg}-T_{\overline{u}f}T_{ug}$$ and
$$H_{u\overline{g}}^*H_{uf}=T_{fg}-T_{\overline{u}g}T_{uf},$$
we have
$$K_1=(T_f T_g-T_gT_f)+(T_{\overline{u}g}T_{uf}-T_{\overline{u}f}T_{ug}),$$
which is a finite sum of finite products of Toeplitz operators.
According to \cite[Theorem 12]{GkZ2}, we obtain by (\ref{compact1}) that $K_1$ is equal to a compact perturbation of a Toeplitz operator, i.e.,
$$K_1=T_h+K$$
for some $h\in L^\infty$ and some compact operator $K$.  Thus $K=K_1-T_h$ belongs to the Toeplitz algebra $\mathscr{T}_{L^{\infty}}$. We conclude by \cite[Corollary 6]{BH} that $h=0$ a.e., which implies that $K_1=K$ is compact.

To show the fourth operator $H_{uf}H_{u\overline{g}}^*+S_fS_g-H_{ug}H_{u\overline{f}}^*-S_gS_f$
is compact, we recall that $$VH_{\varphi}=H_{\varphi}^*V\ \ \ \ \mathrm{and} \ \ \ \  S_{\varphi}V=VT_{\overline{\varphi}}.$$
Then
\begin{equation}\label{3.a}
\begin{array}{l}
V\left(H_{uf}H_{u\overline{g}}^*+S_fS_g-H_{ug}H_{u\overline{f}}^*-S_gS_f\right)V \vspace{2mm}\\
=H_{uf}^*VVH_{u\overline{g}}+T_{\overline{f}}VVT_{\overline{g}}-H_{ug}^*VVH_{u\overline{f}}-T_{\overline{g}}VVT_{\overline{f}} \vspace{2mm}\\
=H_{uf}^*H_{u\overline{g}}+T_{\overline{f}}T_{\overline{g}}-H_{ug}^*H_{u\overline{f}}-T_{\overline{g}}T_{\overline{f}},
\end{array}
\end{equation}
where the second equality follows from $V^2=I$. Using the same method as the above, we can show  similarly that  $$H_{uf}^*H_{u\overline{g}}+T_{\overline{f}}T_{\overline{g}}-H_{ug}^*H_{u\overline{f}}-T_{\overline{g}}T_{\overline{f}}$$ is compact. Furthermore, (\ref{3.a}) gives us that
$$H_{uf}H_{u\overline{g}}^*+S_fS_g-H_{ug}H_{u\overline{f}}^*-S_gS_f$$ is also compact.

Now we turn to the proof of the compactness of the second operator:
$$T_fH_{u\overline{g}}^*+H_{u\overline{f}}^*S_g-T_gH_{u\overline{f}}^*-H_{u\overline{g}}^*S_f.$$  Denoting the above operator by $$K_2=T_fH_{u\overline{g}}^*+H_{u\overline{f}}^*S_g-T_gH_{u\overline{f}}^*-H_{u\overline{g}}^*S_f,$$
we need only to consider the  compactness  of $K_2K_2^*$. From Identity (4.5) in \cite{StZ1}, we have  $$H_{\varphi \psi}=H_{\varphi}T_{\psi}+S_{\varphi}H_{\psi}=H_{\psi}T_{\varphi}+S_{\psi}H_{\varphi}$$ for any $\varphi,\ \psi\in L^\infty$,
to obtain
$$K_2=T_fH_{u\overline{g}}^*-T_{\overline{u}f}H_{\overline{g}}^*-T_gH_{u\overline{f}}^*+T_{\overline{u}g}H_{\overline{f}}^*$$
and $$K_2^*=H_{u\overline{g}}T_{\overline{f}}-H_{\overline{g}}T_{u\overline{f}}-H_{u\overline{f}}T_{\overline{g}}+H_{\overline{f}}T_{u\overline{g}}.$$

Observe that the operator $K_2K_2^*$ is in the Toeplitz algebra $\mathscr{T}_{L^{\infty}}$ and the symbol map maps $K_2K_2^*$ to 0.
By \cite[Theorem 12]{GkZ2} again, we need only to prove that  $$\lim\limits_{z\rightarrow m} \|K_2K_2^*-T_{\phi_z}^*K_2K_2^*T_{\phi_z}\|=0.$$
By Lemma \ref{2g} and $VT_{\varphi\overline{\phi_z}}k_z=-H_{\overline{\varphi}}k_z$ for all $\varphi$ in $L^{\infty}$, $$K_2^*T_{\phi_z}=S_{\phi_z}K_{2}^*-F_z,$$ where $F_z$ is introduced in  Lemma \ref{3a} and $\lim\limits_{|z|\rightarrow 1^-}\|F_z\|=0$.  Thus we have
\begin{align*}
T_{\phi_z}^*K_2K_2^*T_{\phi_z}&= (K_2^*T_{\phi_z})^*K_2^*T_{\phi_z}\\
                               &=(S_{\phi_z}K_{2}^*-F_z)^*(S_{\phi_z}K_{2}^*-F_z)\\
                               &=(K_2S_{\phi_z}^*-F_z^*)(S_{\phi_z}K_{2}^*-F_z)\\
                               &=K_2S_{\phi_z}^*S_{\phi_z}K_{2}^*-K_2S_{\phi_z}^*F_z-F_z^*S_{\phi_z}K_{2}^*+F_z^*F_z\\
                               &=K_2(I-Vk_z\otimes Vk_z)K_{2}^*-K_2S_{\phi_z}^*F_z-F_z^*S_{\phi_z}K_{2}^*+F_z^*F_z\\
                               &=K_2K_2^*-(K_2Vk_z)\otimes (K_2Vk_z)-K_2S_{\phi_z}^*F_z-F_z^*S_{\phi_z}K_{2}^*+F_z^*F_z.
\end{align*}
It follows that $$ K_2K_2^*-T_{\phi_z}^*K_2K_2^*T_{\phi_z}=(K_2Vk_z)\otimes (K_2Vk_z)+K_2S_{\phi_z}^*F_z+F_z^*S_{\phi_z}K_{2}^*-F_z^*F_z.$$
Therefore, in order to show that
$$\lim\limits_{|z|\rightarrow 1^{-}} \|K_2K_2^*-T_{\phi}^*K_2K_2^*T_{\phi}\|=0,$$
it is sufficient to show
\begin{align}\label{3.b}
\lim\limits_{|z|\rightarrow 1^-} \|K_2Vk_z\|_2=0
\end{align}
as $\lim\limits_{|z|\rightarrow 1^-}\|F_z\|=0$.  For this purpose, we will check that each condition of Theorem \ref{MR} can imply (\ref{3.b}).

Recall that $$T_{\varphi}V=VS_{\overline{\varphi}}\ \ \ \ \mathrm{and} \ \ \ \  VH_{\varphi}=H_{\varphi}^*V$$
for all $\varphi\in L^\infty$, we get
\begin{eqnarray}\label{3.c}
K_2Vk_z=V\left(S_{\overline{f}}H_{u\overline{g}}k_z-S_{u\overline{f}}H_{\overline{g}}k_z
+S_{u\overline{g}}H_{\overline{f}}k_z-S_{\overline{g}}H_{u\overline{f}}k_z\right).
\end{eqnarray}

If $f$ and $g$ satisfy  Condition (2) in Theorem \ref{MR}, we have by Lemma \ref{2e} that
$$\lim\limits_{z\rightarrow m}\|H_{\overline{f}}k_z\|_2=\lim\limits_{z\rightarrow m}\|H_{\overline{g}}k_z\|_2=0$$
and $$\lim\limits_{z\rightarrow m}\|H_{u\overline{f}}k_z\|_2=\lim\limits_{z\rightarrow m}\|H_{u\overline{g}}k_z\|_2=0.$$
This gives that $$\lim\limits_{z\rightarrow m}\|K_2Vk_z\|_2=0.$$

Assume that  Condition (1) holds, i.e.,
$$f|_{S_m},\ g|_{S_m},\ \left((u-\lambda)\overline{f}\right)|_{S_m}\ \ \ \ \mathrm{and} \ \ \ \  \left((u-\lambda)\overline{g}\right)|_{S_m}\in H^\infty|_{S_m}.$$
It follows that $$\lim\limits_{z\rightarrow m}\|H_{f}k_z\|_2=\lim\limits_{z\rightarrow m}\|H_{g}k_z\|_2=0$$
and $$\lim\limits_{z\rightarrow m}\|H_{(u-\lambda)\overline{f}}k_z\|_2=\lim\limits_{z\rightarrow m}\|H_{(u-\lambda)\overline{g}}k_z\|_2=0.$$
Computing $K_2Vk_z$ directly, we obtain
\begin{align*}
K_2Vk_z&=V\left(S_{\overline{f}}H_{u\overline{g}}k_z-S_{u\overline{f}}H_{\overline{g}}k_z
+S_{u\overline{g}}H_{\overline{f}}k_z-S_{\overline{g}}H_{u\overline{f}}k_z\right)\\
&=V\left\{S_{\overline{f}}H_{[(u-\lambda)\overline{g}+\lambda\overline{g}]}k_z-S_{u\overline{f}}H_{\overline{g}}k_z
+S_{u\overline{g}}H_{\overline{f}}k_z-S_{\overline{g}}H_{[(u-\lambda)\overline{f}+\lambda\overline{f}]}k_z\right\}\\
&=V\left[S_{\overline{f}}H_{(u-\lambda)\overline{g}}k_z-S_{(u-\lambda)\overline{f}}H_{\overline{g}}k_z
+S_{(u-\lambda)\overline{g}}H_{\overline{f}}k_z-S_{\overline{g}}H_{(u-\lambda)\overline{f}}k_z\right]\\
&=V\left[S_{\overline{f}}H_{(u-\lambda)\overline{g}}k_z-S_{\overline{g}}H_{(u-\lambda)\overline{f}}k_z\right]
+V\left[S_{(u-\lambda)\overline{g}}H_{\overline{f}}k_z-S_{(u-\lambda)\overline{f}}H_{\overline{g}}k_z\right].
\end{align*}
Noting that
\begin{align*}
&S_{(u-\lambda)\overline{g}}H_{\overline{f}}k_z-S_{(u-\lambda)\overline{f}}H_{\overline{g}}k_z\\
&=(I-P)\left[(u-\lambda)\overline{g}(I-P)(\overline{f}k_z)\right]-(I-P)\left[(u-\lambda)\overline{f}(I-P)(\overline{g}k_z)\right]\\
&=(I-P)\left[(u-\lambda)\overline{g}\overline{f}k_z-(u-\lambda)\overline{g}P(\overline{f}k_z)
-(u-\lambda)\overline{f}\overline{g}k_z+(u-\lambda)\overline{f}P(\overline{g}k_z)\right]\\
&=(I-P)\left[(u-\lambda)\overline{f}P(\overline{g}k_z)-(u-\lambda)\overline{g}P(\overline{f}k_z)\right]\\
&=H_{(u-\lambda)\overline{f}}T_{\overline{g}}k_z-H_{(u-\lambda)\overline{g}}T_{\overline{f}}k_z,
\end{align*}
we have
\begin{align*}
K_2Vk_z=&V\left[S_{\overline{f}}H_{(u-\lambda)\overline{g}}k_z-S_{\overline{g}}H_{(u-\lambda)\overline{f}}k_z\right]\\
&+V\left[H_{(u-\lambda)\overline{f}}T_{\overline{g}}k_z-H_{(u-\lambda)\overline{g}}T_{\overline{f}}k_z\right]
\end{align*}
and
\begin{align*}
\|K_2Vk_z\|_2\leqslant  &\|f\|_{\infty}\cdot\|H_{(u-\lambda)\overline{g}}k_z\|_2+\|g\|_{\infty}\cdot\|H_{(u-\lambda)\overline{f}}k_z\|_2\\
&+\|H_{(u-\lambda)\overline{f}}T_{\overline{g}}k_z\|_2+\|H_{(u-\lambda)\overline{g}}T_{\overline{f}}k_z\|_2.
\end{align*}
Since
$$\left((u-\lambda)\overline{g}\right)|_{S_m} \ \ \ \ \mathrm{and} \ \ \ \left((u-\lambda)\overline{f}\right)|_{S_m}\in H^{\infty}|_{S_m},$$
we conclude by Lemma \ref{3.d}  that $\|K_2Vk_z\|_2 \rightarrow 0$ as $z\rightarrow m$.

Finally, we  suppose that Condition (3) holds. Without loss of generality, we assume that $$\left(f-\alpha g\right)|_{S_m}=\beta$$
for some constants $\alpha$ and $\beta$. Then we have
$$\left(f-\alpha g\right)|_{S_m}\ \ \ \  \ \mathrm{and}\ \ \ \ \ \left(\overline{f}-\overline{\alpha}\overline{g}\right)|_{S_m}$$ are in $H^{\infty}|_{S_m}.$ Observe that
\begin{align*}
K_2Vk_z&=V\left(S_{\overline{f}}H_{u\overline{g}}k_z-S_{u\overline{f}}H_{\overline{g}}k_z
+S_{u\overline{g}}H_{\overline{f}}k_z-S_{\overline{g}}H_{u\overline{f}}k_z\right)\\
&=V\left[S_{\overline{f}}H_{u\overline{g}}k_z-S_{u\overline{f}}H_{\overline{g}}k_z
+S_{u\overline{g}}H_{(\overline{f}-\overline{\alpha}\overline{g}+\overline{\alpha}\overline{g})}k_z
-S_{\overline{g}}H_{u(\overline{f}-\overline{\alpha}\overline{g}+\overline{\alpha}\overline{g})}k_z\right]\\
&=V\left[S_{u\overline{g}}H_{(\overline{f}-\overline{\alpha}\overline{g})}k_z-S_{\overline{g}}H_{u(\overline{f}-\overline{\alpha}\overline{g})}k_z\right]
+V\left[S_{(\overline{f}-\overline{\alpha}\overline{g})}H_{u\overline{g}}k_z-S_{u(\overline{f}-\overline{\alpha}\overline{g})}H_{\overline{g}}k_z\right].
\end{align*}
Similarly, we calculate that
\begin{align*}
&S_{(\overline{f}-\overline{\alpha}\overline{g})}H_{u\overline{g}}k_z-S_{u(\overline{f}-\overline{\alpha}\overline{g})}H_{\overline{g}}k_z\\
&=(I-P)\left[(\overline{f}-\overline{\alpha}\overline{g})(I-P)(u\overline{g}k_z)\right]
-(I-P)\left[u(\overline{f}-\overline{\alpha}\overline{g})(I-P)(\overline{g}k_z)\right]\\
&=(I-P)\left[(\overline{f}-\overline{\alpha}\overline{g})u\overline{g}k_z-(\overline{f}-\overline{\alpha}\overline{g})P(u\overline{g}k_z)
-u(\overline{f}-\overline{\alpha}\overline{g})\overline{g}k_z+u(\overline{f}-\overline{x}\overline{g})P(\overline{g}k_z)\right]\\
&=(I-P)\left[u(\overline{f}-\overline{\alpha}\overline{g})P(\overline{g}k_z)-(\overline{f}-\overline{\alpha}\overline{g})P(u\overline{g}k_z)\right]\\
&=H_{u(\overline{f}-\overline{\alpha}\overline{g})}T_{\overline{g}}k_z-H_{(\overline{f}-\overline{\alpha}\overline{g})}T_{u\overline{g}}k_z.
\end{align*}
It follows that
\begin{align*}
\|K_2Vk_z\|_2=&\left\|V\left[S_{u\overline{g}}H_{(\overline{f}-\overline{\alpha}\overline{g})}k_z-S_{\overline{g}}H_{u(\overline{f}-\overline{\alpha}\overline{g})}k_z\right]
+V\left[S_{(\overline{f}-\overline{\alpha}\overline{g})}H_{u\overline{g}}k_z-S_{u(\overline{f}-\overline{\alpha}\overline{g})}H_{\overline{g}}k_z\right]\right\|_2\\
\leqslant & \left\|S_{u\overline{g}}H_{(\overline{f}-\overline{\alpha}\overline{g})}k_z-S_{\overline{g}}H_{u(\overline{f}-\overline{\alpha}\overline{g})}k_z\right\|_2
+\left\|S_{(\overline{f}-\overline{\alpha}\overline{g})}H_{u\overline{g}}k_z-S_{u(\overline{f}-\overline{\alpha}\overline{g})}H_{\overline{g}}k_z\right\|_2\\
\leqslant & \|g\|_{\infty}\cdot\|H_{(\overline{f}-\overline{\alpha}\overline{g})}k_z\|_2+\|g\|_{\infty}\cdot\|H_{u(\overline{f}-\overline{\alpha}\overline{g})}k_z\|_2\\
&+\Big\|H_{u(\overline{f}-\overline{\alpha}\overline{g})}T_{\overline{g}}k_z-H_{(\overline{f}-\overline{\alpha}\overline{g})}T_{u\overline{g}}k_z\Big\|_2\\
\leqslant & \|g\|_{\infty}\cdot\|H_{(\overline{f}-\overline{\alpha}\overline{g})}k_z\|_2+\|g\|_{\infty}\cdot\|H_{u(\overline{f}-\overline{\alpha}\overline{g})}k_z\|_2\\
&+\|H_{u(\overline{f}-\overline{\alpha}\overline{g})}T_{\overline{g}}k_z\|_2+\|H_{(\overline{f}-\overline{\alpha}\overline{g})}T_{u\overline{g}}k_z\|_2.
\end{align*}
Using the conditions that
$$\left(\overline{f}-\overline{\alpha}\overline{g}\right)|_{S_m}\ \ \ \  \mathrm{and}\ \ \ \  \left(u(\overline{f}-\overline{\alpha}\overline{g})\right)|_{S_m}$$ are in $H^{\infty}|_{S_m}$,
we again conclude by Lemma \ref{3.d} that $\|K_2Vk_z\|_2 \rightarrow 0$ as $z\rightarrow m$.

To summarize, each condition in Theorem \ref{MR}  implies
$$\lim\limits_{|z|\rightarrow 1^-} \|K_2K_2^*-T_{\phi}^*K_2K_2^*T_{\phi}\|=0,$$
which gives that $K_2$ is compact.

In order to complete the proof, it remains to show the third operator
$$K_3=H_{uf}T_g+S_fH_{ug}-H_{ug}T_f-S_gH_{uf}$$ is compact.
Rewrite $K_3$ as follows:
\begin{align*}
K_3&=H_{uf}T_g+S_fH_{ug}-H_{ug}T_f-S_gH_{uf}\\
   &=H_{uf}T_g+H_{fug}-H_{f}T_{ug}-H_{ug}T_f-H_{guf}+H_{g}T_{uf}\\
   &=H_{uf}T_g-H_{f}T_{ug}-H_{ug}T_f+H_{g}T_{uf}.
\end{align*}
Observe that
$$K_3^*=T_{\overline{g}}H_{uf}^*-T_{\overline{u}\overline{g}}H_{f}^*-T_{\overline{f}}H_{ug}^*+T_{\overline{u}\overline{f}}H_{g}^*$$
has the same form as $K_2$. Using the same arguments as in the proof of the compactness of $K_2$, we conclude that $K_3^*$ is also compact, which implies that $K_3$ is compact.

Finally, as the necessity part of Theorem \ref{MR} was contained in Theorem \ref{necessary condition}, thus we finish the proof of Theorem \ref{MR}.
\end{proof}

\section{The necessary part of  Theorem \ref{MR2}}
Section 5 is devoted to the proof of the necessary part of Theorem \ref{MR2}. Let  us begin with the following necessary condition for the compactness of the first operator given in Lemma \ref{semi-c}.
\begin{prop}\label{5a}
Let $u$  be a nonconstant inner function, $f,\ g\in L^\infty$ and $m\in \M(H^\infty+C)$. Suppose that
$$T_fT_g+H_{u\overline{f}}^{*}H_{ug}-T_{fg}$$
is compact. Then for the support set $S_m$ of $m$, one of the following holds:\\
 $(1)$ $\overline{f}|_{S_m}$ is in $H^\infty|_{S_m}$; \\
 $(2)$ $g|_{S_m}$ is in $H^\infty|_{S_m}$.
\end{prop}

\begin{proof}
Suppose that
$$K=T_fT_g+H_{u\overline{f}}^{*}H_{ug}-T_{fg}$$
is compact. Clearly,  $K$  can be rewritten as
$$K=H_{u\overline{f}}^{*}H_{ug}-H_{\overline{f}}^*H_{g}.$$
By Lemmas \ref{2a} and \ref{2b}, we have
\begin{eqnarray*}
\lim\limits_{z\rightarrow m}\|K-T_{\phi_z}^{*}KT_{\phi_z}\|
=\lim\limits_{z\rightarrow m}\left\|V\left[H_{u\overline{f}}k_z\otimes H_{ug}k_z-H_{\overline{f}}k_z\otimes H_{g}k_z\right]V^{*}\right\|=0,
\end{eqnarray*}
which gives
\begin{eqnarray}\label{5.a}
\lim\limits_{z\rightarrow m}\left\|H_{u\overline{f}}k_z\otimes H_{ug}k_z-H_{\overline{f}}k_z\otimes H_{g}k_z\right\|=0.
\end{eqnarray}
For $\left[\overline{f}|_{S_m}\right]\in \left(L^{\infty}|_{S_m}\right)\big/\left(H^{\infty}|_{S_m}\right)$, let us consider the following two cases.

\textbf{Case 1.} If $\left[\overline{f}|_{S_m}\right]=0$, then $\overline{f}|_{S_m}\in H^{\infty}|_{S_m}$, as desired.

\textbf{Case 2.} Suppose that $\left[\overline{f}|_{S_m}\right]\neq 0$. Then we have by Lemma \ref{2e} that
$$\varliminf\limits_{z\rightarrow m}\|H_{\overline{f}}k_z\|_2>0.$$
On the other hand,  (\ref{5.a}) gives  that
$$\lim\limits_{z\rightarrow m}\left\|\frac{\langle H_{\overline{f}}k_z,H_{u\overline{f}}k_z\rangle}{\|H_{\overline{f}}k_z\|_2^2} H_{ug}k_z- H_{g}k_z\right\|_2=0.$$
Note that $\frac{\langle H_{\overline{f}}k_z,H_{u\overline{f}}k_z\rangle}{\|H_{\overline{f}}k_z\|_2^2}$ is uniformly bounded for all $z$
in some small neighborhood $\mathcal{O}(m)\cap \D$ of $m$. Using the Bolzano-Weierstrass theorem,  we can find a subnet $\{z_\alpha\}\subset \D$ such that
$$\lim\limits_{z_\alpha\rightarrow m}\frac{\langle H_{\overline{f}}k_{z_{\alpha}},H_{u\overline{f}}k_{z_{\alpha}}\rangle}{\|H_{\overline{f}}k_{z_{\alpha}}\|_2^2}=a$$
for some constant $a$ with $|a|\leqslant 1$. Furthermore, we have
$$\lim\limits_{z_{\alpha}\rightarrow m}\|aH_{ug}k_{z_{\alpha}}-H_{g}k_{z_{\alpha}}\|_2=0.$$
Thus we conclude by  Lemma \ref{2e} that
$$\lim\limits_{z\rightarrow m}\|aH_{ug}k_z-H_{g}k_z\|_2=0,$$
to get
$$\lim\limits_{z\rightarrow m}\|H_{(1-au)g}k_z\|_2=0.$$
Using the same arguments as the one in the proof of Case 2 of Proposition \ref{2c}, we obtain
$$\lim\limits_{z\rightarrow m}\|H_{g}k_z\|_2=0,$$
which implies that $g|_{S_m}\in H^{\infty}|_{S_m}$. This completes the proof.
\end{proof}

The next proposition again follows directly from the following equalities in Remark \ref{rem1}:
$$VT_{\varphi}=S_{\overline{\varphi}}V,\ \ \ VH_{\varphi}=H_{\varphi}^{*}V\ \ \ \mathrm{and} \ \ \ V^2=I.$$
\begin{prop}\label{5b}
Let $u$ be a nonconstant inner function, $f,\ g\in L^\infty$ and $m\in \M(H^\infty+C)$. Assume that
$$H_{uf}H_{u\overline{g}}^{*}+S_{f}S_{g}-S_{fg}$$
is compact. Then for the support set $S_m$ of $m$, one of the following holds:\\
 $(1)$ $f|_{S_m}$ is in $H^\infty|_{S_m}$; \\
 $(2)$ $\overline{g}|_{S_m}$ is in $H^\infty|_{S_m}$.
\end{prop}

Combining Propositions \ref{5a} and \ref{5b}, we obtain a necessary condition for the compactness of $[D_f, D_g)$.
\begin{prop}\label{5c}
Leut $u$ be a nonconstant inner function, $f,\ g\in L^\infty$ and $m\in \M(H^\infty+C)$. Suppose that
the semicommutator $[D_f,D_g)$ is compact. Then for the support set $S_m$ of $m$, one of following conditions holds:\\
 $(1)$ $f|_{S_m}$ and $g|_{S_m}$ are in $H^\infty|_{S_m}$; \\
 $(2)$ $\overline{f}|_{S_m}$ and $\overline{g}|_{S_m}$ are in $H^\infty|_{S_m}$;\\
 $(3)$ either $f|_{S_m}$ or $g|_{S_m}$ is a constant.
\end{prop}
 We establish a necessary condition for the compactness of the operator $T_fH_{u\overline{g}}^{*}+H_{u\overline{f}}^{*}S_g-H_{u\overline{fg}}^{*}$ in the following proposition.
\begin{prop}\label{5d}
Let $u$ be a nonconstant inner function, $f,\ g\in L^\infty$ and $m\in \M(H^\infty+C)$. Suppose that $$T_fH_{u\overline{g}}^{*}+H_{u\overline{f}}^{*}S_g-H_{u\overline{fg}}^{*}$$
is compact and $f|_{S_m},\ g|_{S_m}$ are in $H^{\infty}|_{S_m}$.
Then for the support set $S_m$ of $m$, one of the following holds:\\
 $(1)$ $\left((u-\lambda)\overline{f}\right)|_{S_m}$, $\left((u-\lambda)\overline{g}\right)|_{S_m}$ and
  $\left((u-\lambda)\overline{fg}\right)|_{S_m}$ are in $H^\infty|_{S_m}$ for some constant $\lambda$; \\
 $(2)$ either $f|_{S_m}$ or $g|_{S_m}$ is constant.
\end{prop}
\begin{proof}
Let $K$ denote the compact operator given above, then
$$K^*=H_{u\overline{g}}T_{\overline{f}}+S_{\overline{g}}H_{u\overline{f}}-H_{u\overline{fg}}$$
is also compact.
Using Identity (4.5) of \cite{StZ1}, we obtain
$$H_{u\overline{fg}}=H_{\overline{g}}T_{u\overline{f}}+S_{\overline{g}}H_{u\overline{f}},$$
to get
$$K^*=H_{u\overline{g}}T_{\overline{f}}-H_{\overline{g}}T_{u\overline{f}}.$$
By Lemmas \ref{2g} and \ref{2h}, we obtain that
\begin{equation}\label{5.b}
\begin{array}{l}
\lim\limits_{z\rightarrow m}\left\|K^{*}T_{\phi_z}-S_{\phi_z}K^{*}\right\|
=\lim\limits_{z\rightarrow m}\left\|H_{u\overline{g}}k_z\otimes VH_{\overline{f}}k_z-H_{\overline{g}}k_z\otimes VH_{u\overline{f}}k_z\right\|=0.
\end{array}
\end{equation}

Before going further, we need to consider the following two cases.

\textbf{Case 1.} If $\left[\overline{f}|_{S_m}\right]=0$, then $\overline{f}|_{S_m}\in H^{\infty}|_{S_m}$. Since $f|_{S_m}$ is also in $H^{\infty}|_{S_m}$,
we conclude that $f|_{S_m}$ is a constant.

\textbf{Case 2.} If $\left[\overline{f}|_{S_m}\right]\neq 0$, then we have by Lemma \ref{2e} that
$$\varliminf\limits_{z\rightarrow m}\|H_{\overline{f}}k_z\|_2> 0.$$
By (\ref{5.b}), we have
$$\lim\limits_{z\rightarrow m}\left\|H_{u\overline{g}}k_z-\frac{\langle VH_{\overline{f}}k_z,VH_{u\overline{f}}k_z\rangle}{\|VH_{\overline{f}}k_z\|_{2}^{2}}H_{\overline{g}}k_z\right\|_2=0.$$
Since $V$ is anti-unitary,  $\frac{\langle VH_{\overline{f}}k_z,VH_{u\overline{f}}k_z\rangle}{\|VH_{\overline{f}}k_z\|_{2}^{2}}$
is uniformly bounded for all $z\in \O(m)\cap \D$.
Using the Bolzano-Weierstrass theorem again, there is a subnet $\{z_{\alpha}\}\subset \D$ such that
$$\lim\limits_{z_{\alpha}\rightarrow m}\frac{\langle VH_{\overline{f}}k_{z_{\alpha}},VH_{u\overline{f}}k_{z_{\alpha}}\rangle}{\|VH_{\overline{f}}k_{z_{\alpha}}\|_{2}^{2}}=\lambda$$
for some constant $\lambda$, to obtain
$$\lim\limits_{z_{\alpha}\rightarrow m}\|H_{u\overline{g}}k_{z_{\alpha}}-\lambda H_{\overline{g}}k_{z_{\alpha}}\|_2=0.$$
Now Lemma \ref{2e} gives us that
$$\lim\limits_{z\rightarrow m}\|H_{(u-\lambda)\overline{g}}k_{z}\|_2=0,$$
which implies that $\big((u-\lambda)\overline{g}\big)|_{S_m}\in H^{\infty}|_{S_m}$.

Furthermore, since
\begin{align*}
\|H_{\overline{g}}k_z\otimes VH_{(u-\lambda)\overline{f}}k_z\|
&=\|H_{\overline{g}}k_z\otimes VH_{u\overline{f}}k_z-H_{\lambda\overline{g}}k_z\otimes VH_{\overline{f}}k_z+H_{u\overline{g}}k_z\otimes VH_{\overline{f}}k_z-H_{u\overline{g}}k_z\otimes VH_{\overline{f}}k_z\|\\
&=\left\|H_{(u-\lambda)\overline{g}}k_z\otimes VH_{\overline{f}}k_z-\left(H_{u\overline{g}}k_z\otimes VH_{\overline{f}}k_z-H_{\overline{g}}k_z\otimes VH_{u\overline{f}}k_z\right)\right\|\\
&\leqslant  \|H_{(u-\lambda)\overline{g}}k_z\|_2 \cdot \| VH_{\overline{f}}k_z\|_2+\left\|H_{u\overline{g}}k_z\otimes VH_{\overline{f}}k_z-H_{\overline{g}}k_z\otimes VH_{u\overline{f}}k_z\right\|,
\end{align*}
we conclude that
$$\lim\limits_{z\rightarrow m}\|H_{\overline{g}}k_z\otimes VH_{(u-\lambda)\overline{f}}k_z\|=\lim\limits_{z\rightarrow m}\|H_{\overline{g}}k_z\|_2 \cdot  \|H_{(u-\lambda)\overline{f}}k_z\|_2=0.$$
As $u$ is inner and $f,\ g\in L^\infty$, we obtain that
$$\lim\limits_{z\rightarrow m}\|H_{\overline{g}}k_z\|_2=0 \ \ \ \ \ \mathrm{or}\ \ \ \ \ \lim\limits_{z\rightarrow m}\|H_{(u-\lambda)\overline{f}}k_z\|_2=0.$$
It follows from Lemma \ref{2e} that $\overline{g}|_{S_m}$ or $\left((u-\lambda)\overline{f}\right)|_{S_m}$ is in $H^{\infty}|_{S_m}$.

In order to complete the proof of this proposition, we need to consider the following two subcases for  $\left[\overline{g}|_{S_m}\right]$.

\textbf{Subcase 2(i).} If $\overline{g}|_{S_m}\in H^{\infty}|_{S_m}$, then we have by $g|_{S_m}\in H^{\infty}|_{S_m}$ that $g|_{S_m}$ is a constant.

\textbf{Subcase 2(ii).} If  $\overline{g}|_{S_m}$ is not in $H^{\infty}|_{S_m}$, then we have  $\left((u-\lambda)\overline{f}\right)|_{S_m}\in H^{\infty}|_{S_m}$ and
$$\lim\limits_{z\rightarrow m}\|H_{(u-\lambda)\overline{f}}k_z\|_2=0.$$
Since $K^*$ is compact, we have
$$\lim\limits_{z\rightarrow m}\|K^*k_z\|_2=0.$$
Moreover, we have by Lemma \ref{3.d} that
$$\lim\limits_{z\rightarrow m}\|H_{(u-\lambda)\overline{g}}T_{\overline{f}}k_z\|_2=0.$$
Noting that
\begin{align*}
K^*k_z&=H_{u\overline{g}}T_{\overline{f}}k_z-H_{\overline{g}}T_{u\overline{f}}k_z\\
&=H_{(u-\lambda)\overline{g}}T_{\overline{f}}k_z+H_{\lambda\overline{g}}T_{\overline{f}}k_z-H_{\overline{g}}T_{u\overline{f}}k_z\\
&=H_{(u-\lambda)\overline{g}}T_{\overline{f}}k_z-H_{\overline{g}}T_{(u-\lambda)\overline{f}}k_z\\
&=H_{(u-\lambda)\overline{g}}T_{\overline{f}}k_z-H_{(u-\lambda)\overline{fg}}k_z+S_{\overline{g}}H_{(u-\lambda)\overline{f}}k_z,
\end{align*}
we have $\|H_{(u-\lambda)\overline{fg}}k_z\|_2\rightarrow 0$ as $z\rightarrow m$. Thus
$\left((u-\lambda)\overline{fg}\right)|_{S_m}$ is also in $H^{\infty}|_{S_m}$, to complete the proof of Proposition \ref{5d}.
\end{proof}

In view of Proposition \ref{5d}, we obtain the following proposition which gives a necessary condition for the compactness of the operator $H_{uf}T_g+S_{f}H_{ug}-H_{ufg}$.
\begin{prop}\label{5e}
Let $u$ be a nonconstant inner function, $f,\ g\in L^\infty$ and $m\in \M(H^\infty+C)$. Suppose that
$$H_{uf}T_g+S_{f}H_{ug}-H_{ufg}$$
is compact and $\overline{f}|_{S_m},\ \overline{g}|_{S_m}$ are in $H^{\infty}|_{S_m}$.
Then for the support set $S_m$ of $m$, one of the following holds:\\
 $(1)$ $\left((u-\lambda)f\right)|_{S_m}$, $\left((u-\lambda)g\right)|_{S_m}$ and
  $\left((u-\lambda)fg\right)|_{S_m}$ are in $H^\infty|_{S_m}$ for some constant $\lambda$; \\
 $(2)$ either $f|_{S_m}$ or $g|_{S_m}$ is a constant.
\end{prop}

Combining Propositions \ref{5c}, \ref{5d} and \ref{5e},  now we summarize the necessary condition for the compactness of the semicommutator $[D_f, D_g)$  in the following theorem.
\begin{thm}\label{Suf} Let $u$ be a nonconstant inner function, $f,g\in L^\infty$ and $m\in \M(H^\infty+C)$. Suppose that the semicommutator $[D_f,D_g)$ is compact.  Then for each support set $S_m$ of $m$, one of the following conditions holds: \\
$(1)$ $f|_{S_m}$, $g|_{S_m}$, $((u-\lambda)\overline{f})|_{S_m}$, $\left((u-\lambda)\overline{g}\right)|_{S_m}$ and
$\left((u-\lambda)\overline{fg}\right)|_{S_m}$ are in $H^\infty|_{S_m}$ for some constant $\lambda$; \\
$(2)$ $\overline{f}|_{S_m}$, $\overline{g}|_{S_m}$, $\left((u-\lambda)f\right)|_{S_m}$, $\left((u-\lambda)g\right)|_{S_m}$ and
$\left((u-\lambda)fg\right)|_{S_m}$ are in $H^\infty|_{S_m}$ for some constant $\lambda$; \\
$(3)$ either $f|_{S_m}$ or $g|_{S_m}$ is a constant.
\end{thm}

\section{The sufficient part of  Theorem \ref{MR2}}
In the final section, we will present the proof of the sufficient part of Theorem \ref{MR2}. To do this, we need the following lemma analogous to Lemma \ref{3a}.
\begin{lem}\label{6a}
Let $f,\ g$ be in $ L^\infty$ and
$$L_z=H_{u\overline{g}}k_z\otimes VH_{\overline{f}}k_z-H_{\overline{g}}k_z\otimes VH_{u\overline{f}}k_z,$$
where $z\in \mathbb D$. For each support set $S$, suppose that $f$ and $g$ satisfy one of following conditions:\\
 $(1)$ $f|_S$, $g|_S$, $\left((u-\lambda)\overline{f}\right)|_{S}$, $\left((u-\lambda)\overline{g}\right)|_{S}$ and
 $\left((u-\lambda)\overline{fg}\right)|_{S}$ are in $H^\infty|_{S}$ for some constant $\lambda$; \\
 $(2)$ $\overline{f}|_S$, $\overline{g}|_S$, $\left((u-\lambda)f\right)|_{S}$, $\left((u-\lambda)g\right)|_{S}$ and
 $\left((u-\lambda)fg\right)|_{S}$ are in $H^\infty|_{S}$ for some constant  $\lambda$; \\
 $(3)$ either $f|_S$ or $g|_S$ is constant.\\
Then we have
\begin{align}\label{6.a}
\lim\limits_{|z|\rightarrow 1^-}\|L_z\|=0.
\end{align}
\end{lem}

\begin{proof}
For any $m$ in $\M(H^\infty+C)$, let $S_m$ be the corresponding support set. If Condition $(2)$ or $(3)$ holds, we have by Lemma \ref{2e} that
$$\lim\limits_{z\rightarrow m}\|H_{\overline{f}}k_z\|_2=\lim\limits_{z\rightarrow m}\|H_{u\overline{f}}k_z\|_2=0$$
or
$$\lim\limits_{z\rightarrow m}\|H_{\overline{g}}k_z\|_2=\lim\limits_{z\rightarrow m}\|H_{u\overline{g}}k_z\|_2=0.$$
It follows that $\lim\limits_{|z|\rightarrow m}\|L_z\|=0$.

To finish this proof, we need to show that Condition $(1)$ can imply (\ref{6.a}). By Lemma \ref{2e}, we have
$$\lim\limits_{z\rightarrow m}\|H_{f}k_z\|_2=\lim\limits_{z\rightarrow m}\|H_{g}k_z\|_2=0,$$
$$\lim\limits_{z\rightarrow m}\|H_{(u-\lambda)\overline{f}}k_z\|_2=\lim\limits_{z\rightarrow m}\|H_{(u-\lambda)\overline{g}}k_z\|_2=0$$
and
$$\lim\limits_{z\rightarrow m}\|H_{(u-\lambda)\overline{fg}}k_z\|_2=0.$$
Since
\begin{align*}
\|L_z\|&=\|H_{u\overline{g}}k_z\otimes VH_{\overline{f}}k_z-H_{\overline{g}}k_z\otimes VH_{u\overline{f}}k_z\|\\
   &=\|H_{(u-\lambda)\overline{g}}k_z\otimes VH_{\overline{f}}k_z-H_{\overline{g}}k_z\otimes VH_{(u-\lambda)\overline{f}}k_z\|\\
   &\leqslant\|H_{(u-\lambda)\overline{g}}k_z\otimes VH_{\overline{f}}k_z\|+\|H_{\overline{g}}k_z\otimes VH_{(u-\lambda)\overline{f}}k_z\|\\
   &=\|H_{(u-\lambda)\overline{g}}k_z\|_2 \cdot \|H_{\overline{f}}k_z\|_2+\|H_{\overline{g}}k_z\|_2 \cdot \|H_{(u-\lambda)\overline{f}}k_z\|_2,
\end{align*}
we obtain $\|L_z\|\rightarrow 0$ as $z\rightarrow m$. This completes the proof.
\end{proof}

We are now in  position to prove the sufficiency for Theorem \ref{MR2}.
\begin{proof}[\emph{\textbf{\emph{Proof of the sufficient part of Theorem \ref{MR2}.}}}]
For any $m\in \M(H^\infty+C)$, let $S_m$ be the support set of $m$. Suppose that one of Conditions $(1)$, $(2)$ and $(3)$  in Theorem \ref{MR2} holds.
According to Lemma \ref{semi-c}, we need to show that
$$\widetilde{K_1}=T_fT_g+H_{u\overline{f}}^{*}H_{ug}-T_{fg},$$
$$\widetilde{K_2}=T_fH_{u\overline{g}}^{*}+H_{u\overline{f}}^{*}S_g-H_{u\overline{fg}}^{*},$$
$$\widetilde{K_3}=H_{uf}T_g+S_{f}H_{ug}-H_{ufg}$$
and
$$\widetilde{K_4}=H_{uf}H_{u\overline{g}}^{*}+S_{f}S_{g}-S_{fg}$$
are compact operators.

As  $T_{fg}-T_fT_g=H_{\overline{f}}^*H_g,$ we get
$$\widetilde{K_1}=H_{u\overline{f}}^{*}H_{ug}-H_{\overline{f}}^*H_g.$$
By Lemma \ref{2a}, we have
\begin{equation}\label{6.b}
\begin{array}{l}
\widetilde{K_1}-T_{\phi_z}^*\widetilde{K_1}T_{\phi_z}=V\left[H_{u\overline{f}}k_z\otimes H_{ug}k_z-H_{\overline{f}}k_z\otimes H_gk_z\right]V^*.
\end{array}
\end{equation}
Next we will show that each condition in  Theorem \ref{MR2} can imply that
\begin{align}\label{compact2}
\lim\limits_{z\rightarrow m}\|\widetilde{K_1}-T_{\phi_z}^*\widetilde{K_1}T_{\phi_z}\|=0.
\end{align}

If Condition $(3)$ holds, then we have by Lemma \ref{2e} that
$$\lim\limits_{z\rightarrow m}\|H_{f}k_z\|_2=\lim\limits_{z\rightarrow m}\|H_{g}k_z\|_2=0$$
and
$$\lim\limits_{z\rightarrow m}\|H_{uf}k_z\|_2=\lim\limits_{z\rightarrow m}\|H_{ug}k_z\|_2=0.$$
Observing that
$$\|H_{u\overline{f}}k_z\otimes H_{ug}k_z-H_{\overline{f}}k_z\otimes H_gk_z\|\leqslant\|H_{u\overline{f}}k_z\|_2\cdot \|H_{ug}k_z\|_2+\|H_{\overline{f}}k_z\|_2\cdot \|H_gk_z\|_2,$$
we obtain $$\lim\limits_{z\rightarrow m}\|\widetilde{K_1}-T_{\phi_z}^*\widetilde{K_1}T_{\phi_z}\|=0.$$

Assume that Condition $(1)$ holds. From the proof of the sufficient part of Theorem \ref{MR}, we get that
$$\lim\limits_{z\rightarrow m}\|H_{f}k_z\|_2=\lim\limits_{z\rightarrow m}\|H_{g}k_z\|_2=0,$$
$$\lim\limits_{z\rightarrow m}\|H_{uf}k_z\|_2=\lim\limits_{z\rightarrow m}\|H_{ug}k_z\|_2=0,$$
$$\lim\limits_{z\rightarrow m}\|H_{(u-\lambda)\overline{f}}k_z\|_2=\lim\limits_{z\rightarrow m}\|H_{(u-\lambda)\overline{g}}k_z\|_2=0$$
and
$$\lim\limits_{z\rightarrow m}\|H_{(u-\lambda)\overline{fg}}k_z\|_2=0.$$
Since
$$\|H_{u\overline{f}}k_z\otimes H_{ug}k_z-H_{\overline{f}}k_z\otimes H_gk_z\|\leqslant \|H_{u\overline{f}}k_z\|_2\cdot \|H_{ug}k_z\|_2+\|H_{\overline{f}}k_z\|_2\cdot \|H_gk_z\|_2,$$
 we conclude that $$\lim\limits_{z\rightarrow m}\|\widetilde{K_1}-T_{\phi_z}^*\widetilde{K_1}T_{\phi_z}\|=0.$$

Using the same techniques as above, we can show that Condition $(2)$  implies
$$\lim\limits_{z\rightarrow m}\|\widetilde{K_1}-T_{\phi_z}^*\widetilde{K_1}T_{\phi_z}\|=0.$$
Therefore, each condition of  Theorem \ref{MR2} implies  that
$$\lim\limits_{|z|\rightarrow 1^-}\|\widetilde{K_1}-T_{\phi_z}^*\widetilde{K_1}T_{\phi_z}\|=0.$$

On the other hand, noting
$$H_{u\overline{f}}^{*}H_{ug}=T_{fg}-T_{\overline{u}f}T_{ug},$$
it follows that
$$\widetilde{K_1}=T_fT_g+H_{u\overline{f}}^{*}H_{ug}-T_{fg}=\left(T_{fg}-T_{\overline{u}f}T_{ug}\right)-\left(T_{fg}-T_fT_g\right),$$
which is a finite sum of finite products of Toeplitz operators. Using the same method as in the proof of the sufficient part of Theorem \ref{MR}, we conclude by (\ref{compact2}) that $\widetilde{K_1}$ is compact.

Using
$$VT_{\varphi}=S_{\overline{\varphi}}V,\ \ \ VH_{\varphi}=H_{\varphi}^{*}V\ \ \ \mathrm{and} \ \ \ V^2=I$$
again,  we have
\begin{align*}\label{4.c}
V\widetilde{K_4}V&=V\left(H_{uf}H_{u\overline{g}}^{*}+S_{f}S_{g}-S_{fg}\right)V\\
     &=H_{uf}^*V^2H_{u\overline{g}}+T_{\overline{f}}V^2T_{\overline{g}}-T_{\overline{fg}}V^2\\
     &=H_{uf}^*H_{u\overline{g}}+T_{\overline{f}}T_{\overline{g}}-T_{\overline{fg}}.
\end{align*}
Using the same arguments as above, we conclude that
$$H_{uf}^*H_{u\overline{g}}+T_{\overline{f}}T_{\overline{g}}-T_{\overline{fg}}$$
is compact, which gives us that $\widetilde{K_4}$ is also compact.

To show the compactness of $\widetilde{K_2}$, we will show that $\widetilde{K_2}\widetilde{K_2}^*$ is compact as before. Recall that
$$\widetilde{K_2}=T_fH_{u\overline{g}}^{*}+H_{u\overline{f}}^{*}S_g-H_{u\overline{fg}}^{*}.$$
Using Identity (4.5) in \cite{StZ1} again, we have
$$H_{u\overline{fg}}=S_{\overline{g}}H_{u\overline{f}}+H_{\overline{g}}T_{u\overline{f}}.$$
Thus we get
$$\widetilde{K_2}^*=H_{u\overline{g}}T_{\overline{f}}-H_{\overline{g}}T_{u\overline{f}}$$
and
$$\widetilde{K_2}\widetilde{K_2}^*=(T_fH_{u\overline{g}}^{*}-T_{\overline{u}f}H_{\overline{g}}^*)
(H_{u\overline{g}}T_{\overline{f}}-H_{\overline{g}}T_{u\overline{f}}).$$

Note that $\widetilde{K_2}\widetilde{K_2}^*$ is a finite sum of finite products of Toeplitz operators and the symbol map maps this operator to zero.
Applying  \cite[Lemma 12]{GkZ2} and \cite[Corollary 6]{BH} again, it suffices to show that
$$\lim\limits_{z\rightarrow m}\big\|\widetilde{K_2}\widetilde{K_2}^*-T_{\phi_z}^*\widetilde{K_2}\widetilde{K_2}^*T_{\phi_z}\big\|=0.$$
By Lemma \ref{2g}, we have
\begin{eqnarray*}
\widetilde{K_2}^*T_{\phi_z}=S_{\phi_z}\widetilde{K_2}^*-L_z,
\end{eqnarray*}
where $L_z$ is defined in Lemma \ref{6a}.
Thus we have
\begin{align*}
T_{\phi_z}^*\widetilde{K_2}\widetilde{K_2}^*T_{\phi_z} &=\left(\widetilde{K_2}^*T_{\phi_z}\right)^*\widetilde{K_2}^*T_{\phi_z}\\
                               &=\left(S_{\phi_z}\widetilde{K_2}^*-L_z\right)^*(S_{\phi_z}\widetilde{K_2}^*-L_z)\\
                               &= \left(\widetilde{K_2}S_{\phi_z}^*-L_z^*\right)(S_{\phi_z}\widetilde{K_2}^*-L_z)\\
                               &= \widetilde{K_2}S_{\phi_z}^*S_{\phi_z}\widetilde{K_2}^*-\widetilde{K_2}S_{\phi_z}^*L_z-L_z^*S_{\phi_z}\widetilde{K_2}^*+L_z^*L_z\\
                               &= \widetilde{K_2}(I-Vk_z\otimes Vk_z)\widetilde{K_2}^*-\widetilde{K_2}S_{\phi_z}^*L_z-L_z^*S_{\phi_z}\widetilde{K_2}^*+L_z^*L_z\\
                               &= \widetilde{K_2}\widetilde{K_2}^*-\widetilde{K_2}Vk_z\otimes \widetilde{K_2}Vk_z-\widetilde{K_2}S_{\phi_z}^*L_z-L_z^*S_{\phi_z}\widetilde{K_2}^*+L_z^*L_z.
\end{align*}
Lemma \ref{6a} gives us that $\|\widetilde{K_2}S_{\phi_z}^*L_z\|$, $\|L_z^*S_{\phi_z}\widetilde{K_2}^*\|$ and $\|L_z^*L_z\|$ all converge to $0$  as $z\rightarrow m$. Thus, we need to show that $\|\widetilde{K_2}Vk_z\|_2\rightarrow 0$ as $z\rightarrow m$. In fact,
\begin{align*}
\widetilde{K_2}Vk_z &=(H_{u\overline{g}}T_{\overline{f}}-H_{\overline{g}}T_{u\overline{f}})^*Vk_z\\
        &= T_{f}H_{u\overline{g}}^*Vk_z-T_{\overline{u}f}H_{\overline{g}}^*Vk_z\\
        &= V\left(S_{\overline{f}}H_{u\overline{g}}k_z-S_{u\overline{f}}H_{\overline{g}}k_z\right)\\
        &= V\Big\{(I-P)\left[\overline{f}(I-P)(u\overline{g}k_z)\right]-(I-P)\left[u\overline{f}(I-P)(\overline{g}k_z)\right]\Big\}\\
        &= V(I-P)\left[u\overline{f}P(\overline{g}k_z)-\overline{f}P(u\overline{g}k_z)\right]\\
        &= VH_{u\overline{f}}T_{\overline{g}}k_z-VH_{\overline{f}}T_{u\overline{g}}k_z,
\end{align*}
where the third equality follows from that
$$VT_{\varphi}=S_{\overline{\varphi}}V,\ \ \ VH_{\varphi}=H_{\varphi}^{*}V\ \ \ \mathrm{and} \ \ \ V^2=I.$$

If Condition $(2)$ of Theorem \ref{MR2} holds, then we have
$$\lim\limits_{z\rightarrow m}\|H_{\overline{f}}k_z\|_2=\lim\limits_{z\rightarrow m}\|H_{u\overline{f}}k_z\|_2=0.$$
It follows from Lemma \ref{3.d} that
$$\lim\limits_{z\rightarrow m}\|\widetilde{K_2}Vk_z\|_2=\lim\limits_{z\rightarrow m}\|H_{u\overline{f}}T_{\overline{g}}k_z-H_{\overline{f}}T_{u\overline{g}}k_z\|_2=0.$$

If Condition $(3)$ holds, then  $\overline{f}|_{S_m}$ or $\overline{g}|_{S_m}$ is also a constant. This   yields
$$\lim\limits_{z\rightarrow m}\|H_{\overline{f}}k_z\|_2=\lim\limits_{z\rightarrow m}\|H_{u\overline{f}}k_z\|_2=0$$
or
$$\lim\limits_{z\rightarrow m}\|H_{\overline{g}}k_z\|_2=\lim\limits_{z\rightarrow m}\|H_{u\overline{g}}k_z\|_2=0.$$
By Lemma \ref{3.d} again, we have
$$\lim\limits_{z\rightarrow m}\|\widetilde{K_2}Vk_z\|_2=\lim\limits_{z\rightarrow m}\|H_{u\overline{f}}T_{\overline{g}}k_z-H_{\overline{f}}T_{u\overline{g}}k_z\|_2=0$$
 or
$$\lim\limits_{z\rightarrow m}\|\widetilde{K_2}Vk_z\|_2=\lim\limits_{z\rightarrow m}\|S_{\overline{f}}H_{u\overline{g}}k_z-S_{u\overline{f}}H_{\overline{g}}k_z\|_2=0.$$

Finally,  we  assume that Condition $(1)$ holds. From Lemma \ref{2e}, we get
$$\lim\limits_{z\rightarrow m}\|H_{(u-\lambda)\overline{f}}k_z\|_2=\lim\limits_{z\rightarrow m}\|H_{(u-\lambda)\overline{g}}k_z\|_2=\lim\limits_{z\rightarrow m}\|H_{(u-\lambda)\overline{fg}}k_z\|_2=0.$$
Noting that
\begin{align*}
\|\widetilde{K_2}Vk_z\|_2&=\|H_{u\overline{f}}T_{\overline{g}}k_z-H_{\overline{f}}T_{u\overline{g}}k_z\|_2\\
              &=\|H_{(u-\lambda)\overline{f}}T_{\overline{g}}k_z-H_{\overline{f}}T_{(u-\lambda)\overline{g}}k_z\|_2\\
              &= \|H_{(u-\lambda)\overline{f}}T_{\overline{g}}k_z-[H_{(u-\lambda)\overline{fg}}-S_{\overline{f}}H_{(u-\lambda)\overline{g}}]k_z\|_2\\
              &=\|H_{(u-\lambda)\overline{f}}T_{\overline{g}}k_z-H_{(u-\lambda)\overline{fg}}k_z+S_{\overline{f}}H_{(u-\lambda)\overline{g}}k_z\|_2\\
              &\leqslant \|H_{(u-\lambda)\overline{f}}T_{\overline{g}}k_z\|_2+\|H_{(u-\lambda)\overline{fg}}k_z\|_2+\|S_{\overline{f}}H_{(u-\lambda)\overline{g}}k_z\|_2\\
              &\leqslant\|H_{(u-\lambda)\overline{f}}T_{\overline{g}}k_z\|_2+\|H_{(u-\lambda)\overline{fg}}k_z\|_2+\|f\|_{\infty}\cdot\|H_{(u-\lambda)\overline{g}}k_z\|_2,
\end{align*}
we conclude by Lemma \ref{3.d} that $\lim\limits_{z\rightarrow m}\|\widetilde{K_2}Vk_z\|_2=0$. Moreover, since
\begin{align*}
\|\widetilde{K_2}\widetilde{K_2}^*-T_{\phi_z}^*\widetilde{K_2}\widetilde{K_2}^*T_{\phi_z}\|&=\|\widetilde{K_2}Vk_z\otimes \widetilde{K_2}Vk_z+\widetilde{K_2}S_{\phi_z}^*L_z+L_z^*S_{\phi_z}\widetilde{K_2}^*-L_z^*L_z\|\\
                                         &\leqslant \|\widetilde{K_2}Vk_z\otimes \widetilde{K_2}Vk_z\|+\|\widetilde{K_2}S_{\phi_z}^*L_z\|+\|L_z^*S_{\phi_z}\widetilde{K_2}^*\|+\|L_z^*L_z\|\\
                                         &=\|\widetilde{K_2}Vk_z\|_2^2+\|\widetilde{K_2}S_{\phi_z}^*L_z\|+\|L_z^*S_{\phi_z}\widetilde{K_2}^*\|+\|L_z\|^2,
\end{align*}
we have
$$\lim_{z\rightarrow m}\|\widetilde{K_2}\widetilde{K_2}^*-T_{\phi_z}\widetilde{K_2}\widetilde{K_2}^*T_{\phi_z}\|=0.$$
Using the same idea as in the proof of the compactness of $\widetilde{K_1}$, we conclude that $\widetilde{K_2}\widetilde{K_2}^*$ is compact, so $\widetilde{K_2}$ is also compact.

In order to finish  the proof, we observe that
\begin{align*}
V\widetilde{K_3}V&=VH_{uf}T_gV+VS_{f}H_{ug}V-VH_{ufg}V\\
      &=H_{uf}^*S_{\overline{g}}+T_{\overline{f}}H_{ug}^*-H_{ufg}^*.
\end{align*}
Similarly we can show that $\widetilde{K_3}$ is compact, to complete the proof of Theorem \ref{MR2}.
\end{proof}\vspace{3.6mm}
\subsection*{Acknowledgment}
This work was partially supported by  NSFC (grant numbers: 11531003,  11701052). The second author was partially supported by the Fundamental Research Funds for the Central Universities (grant numbers: 2020CDJQY-A039, 2020CDJ-LHSS-003).


\begin{thebibliography}{99}

   \bibitem{ACS} S. Axler, S-Y. A. Chang and D. Sarason, Products of Toeplitz operators,  \emph{Integral Equations Operator Theory}, 1 (1978), no. 3, 285-309.

    \bibitem{AC} S. Axler and \v{Z}. \v{C}u\v{c}kvoi\'{c}, Commuting Toeplitz operators with harmonic symbols,  \emph{Integral Equations Operator Theory}, 14 (1991), no. 1, 1-12.


    \bibitem{BH} J. Barr\'{\i}a and P. R. Halmos, Asymptotic Toeplitz operators,  \emph{Trans. Amer. Math. Soc.}, 273 (1982), no. 2, 621-630.

    \bibitem{BrHa} A. Brown and  P. R. Halmos, Algebraic properties of Toeplitz operators,  \emph{J. Reine. Angew. Math.}, 213 (1964), no. 1, 89-102.

    \bibitem{Ca} C. C\^{a}mara, K. Kli\'{s}-Garlicka, B. {\L}anucha and M. Ptak, Invertibility, Fredholmness
and kernels of dual truncated Toeplitz operators, \emph{Banach J. Math. Anal.} 14
(2020), no. 4, 1558-1580.

    \bibitem{DZ} X. Ding, S. Sun and D. Zheng, Commuting Toeplitz operators on the bidisk,  \emph{J. Funct. Anal.}, 263 (2012), no. 11, 3333-3357.

    \bibitem{DS} X. Ding and Y. Sang, Dual truncated Toeplitz operators,  \emph{J. Math. Anal. Appl.}, 461 (2018), no. 1, 929-946.

    \bibitem{Dou} R. Douglas,  \emph{Banach Algebra Techniques in Operator Theory}, second edition, Graduate Texts in Mathematics, vol. 179, Springer, New York, 1998.

    \bibitem{Dur} P. L. Duren, \emph{Theory of $H^p$ Spaces}, Academic Press, New York,  2000.

    \bibitem{G}   S. R. Garcia, J. Mashreghi and  W. T. Ross, \emph{Introduction to Model Spaces and Their Operators}, Cambridge University Press, 2016.

    \bibitem{Gar} J. B. Garnett, \emph{Bounded Analytic Functions},  Academic Press,  New York, 1981.

    \bibitem{GpZ} P. Gorkin and D. Zheng, Essentially commuting Toeplitz operators,  \emph{Pacific J. Math.}, 190 (1999), no. 1, 87-109.

    \bibitem{Gu} C. Gu and D. Zheng, The semicommutator of Toeplitz operators on the bidisc, \emph{ J. Operator Theory}, 38 (1997), no. 1, 173-193.

    \bibitem{GkZ1} K. Guo and D. Zheng, Essentially commuting Hankel and Toeplitz operators,  \emph{J. Funct. Anal.}, 201 (2003), no. 1, 121-147.

    \bibitem{GkZ2} K. Guo and D. Zheng, The distribution function inequality for a finite sum of finite products of Toeplitz operators,  \emph{J. Funct. Anal.}, 218 (2005), no. 1, 1-53.


    \bibitem{GSZ} K. Guo, S. Sun and D. Zheng, Finite rank commutators and semicommutators of Toeplitz operators with harmonic symbols,
   \emph{Illinois J. Math.},  51 (2007), no. 2, 583-596.



    \bibitem{Hof} K. Hoffman, \emph{Banach Spaces of Analytic Functions}, Englewood Cliffs, 1962.

    \bibitem{DSQ} Y. Sang, Y. Qin and X. Ding, Dual truncated Toeplitz $C^*$-algebras,  \emph{Banach J. Math. Anal.}, 13 (2019), no. 2, 275-292.

    \bibitem{SQD} Y. Sang, Y. Qin and X. Ding, A theorem of Brown-Halmos type for dual truncated Toeplitz operators,  \emph{Ann. Funct. Anal.}, 11 (2020), no. 1, 271-284.

    \bibitem{Str} K. Stroethoff, Essentially commuting Toeplitz operators with harmonic symbols,  \emph{ Canad. J. Math.}, 45 (1993), no. 5, 1080-1093.

    \bibitem{StZ1} K. Stroethoff and D. Zheng,  Products of Hankel and Toeplitz operators on the Bergman space, \emph{J. Funct. Anal.}, 169 (1999), no. 1, 289-313.


    \bibitem{StZ} K. Stroethoff and D. Zheng, Algebraic and spectral properties of dual Toeplitz operators,  \emph{Trans. Amer. Math. Soc.}, 354 (2002), no. 6, 2495-2520.

    \bibitem{V} A. Volberg, Two remarks concerning the theorem of S. Axler, S.-Y. A. Chang and D. Sarason, \emph{J. Operator Theory}, 7 (1982), no. 2, 209-218.

    \bibitem{Zheng} D. Zheng, The distribution function inequality and products of Toeplitz operators and Hankel operators,  \emph{J. Funct. Anal.}, 138 (1996), no. 2, 447-501.

    \bibitem{Zheng1} D. Zheng, Toeplitz operators and Hankel operators on the Hardy space of the unit sphere,  \emph{J. Funct. Anal.},  149 (1997),  no. 1, 1-24.

    \bibitem{Zhu} K. Zhu, \emph{Operator Theory in Function Spaces}, second edition,  American Mathematical Society, Providence, RI, 2007.

\end{thebibliography}
\end{document}